\documentclass{article}
\usepackage{amsfonts}
\usepackage{amsmath}
\usepackage{amssymb}
\usepackage{multirow}
\usepackage{lscape}
\input xy \xyoption{all}

\usepackage{easy-todo}
\usepackage{rotating}
\usepackage{hyperref}
\usepackage{color}

\newtheorem{theorem}{Theorem}[section]

\newtheorem{corollary}[theorem]{Corollary}

\newtheorem{definition}[theorem]{Definition}
\newtheorem{example}[theorem]{Example}

\newtheorem{lemma}[theorem]{Lemma}

\newtheorem{proposition}[theorem]{Proposition}

\newtheorem{remark}[theorem]{Remark}

\newenvironment{proof}[1][Proof]{\noindent\textbf{#1.} }{\ \rule{0.5em}{0.5em}}

\newcommand{\Cyclic}{\mathop{\mathchoice%
  {\vcenter{\hbox{\LARGE\( \mathfrak S \)%
    \vrule width 0pt height 1.7ex depth 0.2ex}}}
  {\vcenter{\hbox{\Large\( \mathfrak S \)\kern-0.1em}}}
  {\vcenter{\hbox{\normalsize\( \mathfrak S \)\kern -0.1em}}}
  {\vcenter{\hbox{\scriptsize\( \mathfrak S \)\kern -0.1em}}}}}

\newcommand{\f}{\mathfrak}

\newcommand{\dt}{\left.\frac{d}{dt}\right|_{t=0}}

\newcommand{\wnabla}{\widetilde{\nabla}}

\newcommand{\R}{\mathbb{R}}
\newcommand{\C}{\mathbb{C}}

\newcommand{\SXYZ}{\raisebox{-0.6ex}{\mbox{\scriptsize{$\mathit{XYZ}$}}}}
\newcommand{\SYZV}{\raisebox{-0.6ex}{\mbox{\scriptsize{$\mathit{YZV}$}}}}

\newcommand{\Lie}[1]{\textsl{#1}}

\title{Reductive locally homogeneous pseudo-Riemannian manifolds and Ambrose-Singer connections}

\author{ Ignacio Luj\'an  \\
\small{Departamento de Geometr\'\i a y Topolog\'\i a} \\
\small{Facultad de Matem\'aticas, Universidad Complutense de Madrid}\\
\small{28040 Madrid, Spain}}
\date{}

\begin{document}

\maketitle

\abstract{Ambrose and Singer characterized connected,
simply-connected and complete homogeneous Riemannian manifolds as
Riemannian manifolds admitting a metric connection such that its
curvature and torsion are parallel.  The aim of this paper is to
extend Ambrose-Singer Theorem to the general framework of locally
homogeneous pseudo-Riemannian manifolds. In addition we study
under which conditions a locally homogeneous pseudo-Riemannian
manifold can be recovered from the curvature and their covariant
derivatives at some point up to finite order. The same problem is
tackled in the presence of a geometric structure.}

\renewcommand{\thefootnote}{\fnsymbol{footnote}}
\renewcommand{\thefootnote}{\arabic{footnote}}
\renewcommand{\thefootnote}{\fnsymbol{footnote}}
\footnotetext{pseudo-Riemannian manifolds,
invariant geometric structures.}
\renewcommand{\thefootnote}{\arabic{footnote}}


\section{Introduction}

In \cite{Ambrose-Singer} Ambrose and Singer characterized
connected, simply-connected and complete homogeneous Riemannian
manifolds as Riemannian manifolds $(M,g)$ admitting a linear
connection $\wnabla$ satisfying
$$\wnabla g=0,\qquad \wnabla
R=0,\qquad \wnabla S=0,$$ where $S=\nabla-\wnabla$, $\nabla$ is
the Levi-Civita connection of $g$, and $R$ the curvature tensor
field of $g$. Connections satisfying the previous equations would
become later known as Ambrose-Singer connections. Since their
introduction, Ambrose-Singer connections have become an
extensively used tool for the study of homogeneity. In addition,
Ambrose-Singer Theorem has been extended to the case when the
manifold is endowed with a geometric structure \cite{Kir}, and
later to the pseudo-Riemannian setting \cite{GO}.

Regarding locally homogeneous spaces, the following result is
known (see for instance \cite{Tri}).

\begin{theorem}\label{Thm AS local riemanniano}
Let $(M,g)$ be a Riemannian manifold. Then $(M,g)$ is locally
homogeneous if and only if it admits an Ambrose-Singer connection.
\end{theorem}

This theorem is no longer true if $g$ is a metric with signature.
In fact, in \cite{GO} it was proved that if a globally homogeneous
pseudo-Riemannian manifold admits an Ambrose-Singer connection,
then it must be reductive. This suggests that in order to extend
Ambrose-Singer Theorem to the pseudo-Riemannian setting we have to
find a condition playing the same role as the reductivity
condition for globally homogeneous spaces.

The aim of this paper is to formulate and prove an analogous
result to Theorem \ref{Thm AS local riemanniano} for
pseudo-Riemannian manifolds. The case when an invariant geometric
structure is present is also analyzed.

On the other hand, all the proofs of Theorem \ref{Thm AS local
riemanniano} (known by the author) make use of the so called
``canonical'' Amborse-Singer connection constructed by Kowalski
\cite{Ko2}. The construction of this connection relies strongly on
the fact that the Killing form of $\f{so}(T_pM)$ is definite if
the metric $g$ is Riemannian, so that a straightforward adaptation
to the pseudo-Riemannian realm is not possible. In Section
\ref{section strongly reductive} we show, under suitable
conditions, how to adapt the construction of the ``canonical''
Ambrose-Singer connection made by Kowalski to metrics with
signature. This will lead to a new notion of reductivity called
``strong reductivity''. As a consequence we will see that strongly
reductive locally homogeneous pseudo-Riemannian manifolds can be
recovered from the curvature and their covariant derivatives at
some point up to finite order. Recall that this property is known
to be satisfied by all locally homogeneous Riemannian manifolds
(see \cite{Nic-Tri}). An analogous result will hold in the
presence of an invariant geometric structure.

\section{Preliminaries}

For a comprehensive introduction on Lie pseudo-groups and
transitive Lie algebras see \cite{Spiro} and the references
therein. We just recall that a transitive Lie algebra is a pair
$(L,L^0)$, where $L$ is a Lie algebra and $L^0$ is a proper
subalgebra such that the only ideal of $L$ contained in $L^0$ is
$\{0\}$.

Let $(M,g)$ be a locally homogeneous pseudo-Riemannian space, and
let $\mathcal{I}$ denote the Lie pseudo-group of local isometries
acting transitively on $(M,g)$. The system of PDE's that must be
satisfied by the elements of $\mathcal{I}$ is
$$f^*g=g.$$
The corresponding system of Lie equations is thus
\begin{equation}\label{Lie equation}\mathcal{L}_Xg=0,\end{equation}
that is, infinitesimal transformations are given by local Killing
vector fields. For a fixed point $p\in M$ we choose a basis
$\{e_1,\ldots,e_m\}$ of $T_pM$. The set $\{e^1,\ldots,e^m\}$
denotes its dual basis. We consider the transitive Lie algebra
$(\f{i},\f{i}^0)$ associated with the system \eqref{Lie equation}.
The Lie algebra $\f{i}$ is the set of vector valued formal power
series
$$\xi=\sum_{r,i,j_1,\ldots,j_r}\xi^i_{j_1\ldots j_r}e_i\otimes e^{j_1}\odot\ldots\odot e^{j_r},$$
where $\xi^i_{j_1\ldots j_r}$ solve \eqref{Lie equation} and all
its derivatives. The subalgebra $\f{i}^0$ is formed by all the
elements of $\f{i}$ such that the terms $\xi^i$ of order zero
vanish. As seen in \cite{Spiro}, an element $\xi\in\f{i}$ is
determined by the terms of order $0$ and $1$, which lie in $T_pM$
and $\f{so}(T_pM)$ respectively.

\begin{definition}
A Killing generator at $p$ is a pair $(X,A)\in
T_pM\times\f{so}(T_pM)$ verifying
$$A\cdot \nabla^iR_p+i_X\nabla^{i+1}R_p=0,\qquad i\geq 0,$$
where $\nabla$ is the Levi-Civita connection of $g$, and $R$ its
curvature tensor field.
\end{definition}

The set $\f{kill}$ of Killing generators at $p$ has a Lie algebra
structure with bracket
$$[(X,A),(Y,B)]=(AX-BY,(R_p)_{XY}+[A,B]).$$
We define
$$\f{kill}^0=\{(X,A)\in\f{kill}/\, X=0\}.$$

\begin{lemma}\label{lemma isomorphism with Killing gen}\cite{Spiro}
$(\f{kill},\f{kill}^0)$ is a transitive Lie algebra isomorphic to
$(\f{i},\f{i}^0)$.
\end{lemma}

\begin{proof}
Let $(x^1,\ldots,x^m)$ be a set of normal coordinates around $p$.
We consider the map
$$\begin{array}{rcl}
\f{i} & \to & \f{kill}\\
(\xi^i,\xi^i_j) & \mapsto &
(\xi^i\partial_{x^i|p},\xi^i_j\partial_{x^i|p}\otimes dx^j_{|p}),
\end{array}$$
where $(\xi^i,\xi^i_j)$ are the terms of order $0$ and $1$
characterizing an element $\xi\in \f{i}$. As a straightforward
computation shows, this map defines a Lie algebra homomorphism.
\end{proof}

\bigskip

Let now $\xi$ be a local vector field on $M$, we define the
$(1,1)$-tensor field
$$A_{\xi}=\mathcal{L}_{\xi}-\nabla_{\xi}=-\nabla \xi,$$
where $\mathcal{L}$ denotes Lie derivative. Among the equations
that $\xi$ must satisfy at $p$, we have
$$(\mathcal{L}_{\xi}g)_p=0,\hspace{1em}(\mathcal{L}_{\xi}\nabla^iR)_p=0,\hspace{1em} i\geq 0,$$
which coincide with
$$A\cdot g_p=0,\hspace{1em}A\cdot \nabla^iR_p+i_{X}\nabla^{i+1}R_p=0,\hspace{1em} i\geq 0,$$
for $X=\xi_p$ and $A=A_{\xi}|_{p}$, whence $(\xi_p,A_{\xi}|_{p})$
is a Killing generator.

\begin{corollary}\label{corollary realized germs}
Every formal solution $\xi\in\f{i}$ is realized by the germ of a
local Killing vector field.
\end{corollary}

\begin{proof}
Adapting the arguments used by Nomizu in \cite{Nomizu} to metrics
with signature, we see that if the dimension of the Lie algebra of
Killing generators is constant on $M$, then for every Killing
generator $(X,A)$ at a point $p$, there exist a local Killing
vector field $\xi$ with $(X,A)=(\xi_p,A_{\xi}|_{p})$.
\end{proof}

\bigskip

The Lie algebra isomorphism exhibited in the proof of Lemma
\ref{lemma isomorphism with Killing gen} can be seen as
$$\begin{array}{rcl}
\f{i} & \to & \f{kill}\\
\left[ \xi \right] & \mapsto & (\xi_p,A_{\xi}|_{p}),
\end{array}$$
where $[\xi]$ denotes the germ of the local vector field $\xi$ at
$p$.

\section{Reductive locally homogeneous\\ pseudo-Riemannian manifolds}

We now consider a Lie pseudo-group $\mathcal{G}\subset\mathcal{I}$
acting transitively on $(M,g)$. A Lie subalgebra
$\f{g}\subset\f{i}$ can be attached to $\mathcal{G}$, namely
$\f{g}$ is the set of germs of local Killing vector fields with
$1$-parameter group contained in $\mathcal{G}$. The Lie algebra
$\f{k}$ formed by those $[\xi]\in\f{g}$ vanishing at $p$ is thus a
Lie subalgebra of $\f{i}^0$, and the pair $(\f{g},\f{k})$ is a
transitive Lie algebra.

\begin{definition}\label{definition linear isotropy group}
Let $\mathcal{G}$ be a Lie pseudo-group of isometries acting
transitively on $(M,g)$. The isotropy pseudo-group at a point $p$
is
$$\mathcal{H}_p=\{f\in\mathcal{G}/\, f(p)=p\}.$$
\end{definition}

Since $f(p)=p$ is not a differential equation, $\mathcal{H}_p$ is
not a Lie pseudo-group in general. For this reason it is more
convenient to work with the so called linear isotropy group.

\begin{definition}
The linear isotropy group of $\mathcal{G}$ at $p\in M$ is
$$H_p=\{F:T_pM\to T_pM /\, F=f_*,\,f\in\mathcal{H}_p\}.$$
\end{definition}

Since every $f\in\mathcal{H}_p$ is an isometry, $H_p$ is a Lie
subgroup of $\Lie{O}(T_pM)$.

\begin{lemma}
The Lie algebra $\f{h}_p$ of $H_p$ is isomorphic to $\f{k}$.
\end{lemma}

\begin{proof}
We define the map
$$\begin{array}{rcl}
\f{k} & \to & \f{h}_p\\
\left[ \xi\right] & \mapsto & \dt (f_t)_*,
\end{array}$$
where $f_t\subset \mathcal{H}_p$ is the $1$-parameter group
generated by $\xi$. A simple inspection shows that this map is a
Lie algebra isomorphism.
\end{proof}

\bigskip

Note that seeing $\f{h}_p$ as a subalgebra of $\f{so}(T_pM)$, the
previous isomorphism between $\f{k}$ and $\f{h}_p$ can be read as
$$\begin{array}{rcl}
\f{k} & \to & \f{h}_p\\
\left[ \xi\right] & \mapsto & A_{\xi}|_{p}.
\end{array}$$
There is a natural action of $H_p$ on $\f{g}$ given by
$$\begin{array}{rrcl}
\mathrm{Ad}: & H_p\times \f{g} & \to & \f{g}\\
             & (F,\left[ \xi\right]) & \mapsto & \left[\eta\right],
\end{array}$$
with
$$\eta_q=\dt f\circ \varphi_t\circ f^{-1}(q),$$
for every $q$ in a certain neighborhood of $p$, where $\varphi_t$
is the $1$-parameter group generated by $[\xi]$, and $F=f_*$. When
identifying $\f{k}$ with $\f{h}_p$, the restriction of this action
to $\f{k}$ is just the usual adjoint action of $H_p$ on its Lie
algebra.

\begin{definition}
Let $(M,g)$ be a pseudo-Riemannian manifold, and let $\mathcal{G}$
be a Lie pseudo-group of isometries acting transitively on
$(M,g)$. We say that the triple $(M,g,\mathcal{G})$ is reductive
if the transitive Lie algebra $(\f{g},\f{k})$ associated with
$\mathcal{G}$ can be decomposed as $\f{g}=\f{m}\oplus \f{k}$,
where $\f{m}$ is $\mathrm{Ad}(H_p)$-invariant.
\end{definition}

Note that being reductive is a property of the triple
$(M,g,\mathcal{G})$ rather than a property of the
pseudo-Riemannian manifold $(M,g)$ itself. In Section \ref{section
reductivity condition} we will show that the same locally
homogeneous pseudo-Riemannian manifold can be reductive for the
action of a certain Lie pseudo-group $\mathcal{G}$, whereas it is
non-reductive for the action of another Lie pseudo-group
$\mathcal{G}'$. On the other hand, it seems that the previous
definition depends on the chosen point $p\in M$, however

\begin{proposition}
If $(M,g,\mathcal{G})$ is reductive at a point $p\in M$, then it
is reductive at every point.
\end{proposition}

\begin{proof}
Let $q$ be another point of $M$. We denote by $(\f{g}_p,\f{k}_p)$
and $(\f{g}_q,\f{k}_q)$ the transitive Lie algebras associated
with $\mathcal{G}$ at $p$ and $q$ respectively. Let $h\in
\mathcal{G}$ be a local isometry with $h(p)=q$. $h$ induces
isomorphisms $\hat{h}:\f{g}_p\to \f{g}_q$, $[\xi]\mapsto
[h_*(\xi)]$, and $\check{h}:H_p\to H_p$, $F\mapsto h_*\circ F\circ
h_*^{-1}$. Let $\f{g}_p=\f{m}_p\oplus\f{k}_p$ with $\f{m}_p$
$\mathrm{Ad}(H_p)$-invariant, we define
$\f{m}_q=\hat{h}(\f{m}_p)\subset\f{g}_q$. It is obvious that
$\f{g}_q=\f{m}_q\oplus\f{k}_q$, since $\hat{h}$ is an isomorphism
and takes $\f{k}_p$ to $\f{k}_q$. We now show that $\f{m}_q$ is
$\mathrm{Ad}(H_q)$-invariant and independent of the local isometry
$h$. Let $F\in H_q$, and let $f\in \mathcal{H}_q$ with $F=f_*$.
Let $[\eta]\in \f{m}_q$, there is an element $[\xi]\in \f{m}_p$
with $\eta=h_*(\xi)$. The $1$-parameter group generated by $\eta$
is thus $\phi_t=h\circ \varphi_t\circ h^{-1}$, where $\varphi_t$
is the $1$-parameter group generated by $\xi$. Therefore
\begin{align*}
\mathrm{Ad}_F([\eta]) & = \left[\dt f\circ \phi_t\circ
f^{-1}\right]\\
 & = \left[\dt f\circ h\circ \varphi_t\circ h^{-1}\circ
 f^{-1}\right]\\
 & = \left[\dt h\circ h^{-1}\circ f\circ h\circ \varphi_t\circ h^{-1}\circ
 f^{-1}\circ  h\circ h^{-1}\right]\\
 & = h_*\left(\mathrm{Ad}_{\check{h}^{-1}(F)}([\xi])\right).
\end{align*}
Since $\check{h}^{-1}(F)\in H_p$, we have
$\mathrm{Ad}_F([\eta])\in\f{m}_q$. On the other hand, in order to
prove the independence of $h$, it is enough to prove that for
other $h'\in\mathcal{G}$ with $h'(p)=q$ we have that
$h_*^{-1}\circ h'_*([\xi])\in \f{m}_p$. But
$$h^{-1}_*\circ h'_*([\xi])  = \left[\dt h^{-1}\circ h'\circ
\varphi_t\right]=\mathrm{Ad}_{(h^{-1}\circ h')_*}([\xi]).$$ Since
$h^{-1}\circ h'\in\mathcal{H}_p$ and $\f{m}_p$ is
$\mathrm{Ad}(H_p)$-invariant we conclude that $h^{-1}_*\circ
h'_*([\xi])\in\f{m}_p$.
\end{proof}

\bigskip

Following \cite{Ambrose-Singer} we give the following definition.

\begin{definition}
An Ambrose-Singer connection (AS-connection for sort), is a linear
connection $\wnabla$ satisfying
\begin{equation*}
\wnabla g=0,\qquad \wnabla R=0,\qquad \wnabla S=0,
\end{equation*}
where $S=\wnabla-\nabla$.
\end{definition}

The following two theorems characterize locally homogeneous
pseudo-Rie\-mannian manifolds admitting an AS-connection.

\begin{theorem}\label{theorem 1}
Let $(M,g,\mathcal{G})$ be a reductive locally homogeneous
pseudo-Rie\-mannian manifold. Then $(M,g)$ admits an
AS-connection.
\end{theorem}

\begin{proof}
Let $(r,s)$ be the signature of $g$, and let $\mathcal{O}(M)$ be
the bundle of orthonormal references of $M$. We fix a point $p\in
M$ and a reference $u_0\in\mathcal{O}(M)$ in the fiber of $p$. We
shall interpret an orthonormal reference $u$ at $q\in M$ as an
isometry $u:(\R^{m},\langle\,,\,\rangle)\to (T_qM,g_q)$, where
$\langle\,,\,\rangle$ is the standard metric of $\R^m$ with
signature $(r,s)$. Consider the set
\begin{equation}\label{bundle Q}Q=\{u\in\mathcal{O}(M)/\, u=\widetilde{h}(u_0),\, h\in
\mathcal{G}\},\end{equation} where $\widetilde{h}$ is the map
induced on $\mathcal{O}(M)$ by a local isometry $h$. $Q$
determines a reduction of $\mathcal{O}(M)$ with structure group
$$\bar{H}=\{B\in\Lie{O}(r,s)/\, \hat{u}_0(B)=f_*,\, f\in \mathcal{H}_p\},$$
where $\hat{u}_0:\Lie{O}(r,s)\to \Lie{O}(T_pM)$, $B\mapsto
u_0\circ B\circ u_0^{-1}$. It is obvious that $\hat{u}_0$ gives an
isomorphism between $\bar{H}$ and the linear isotropy group $H_p$.
The right action of an element $B\in\bar{H}$ on a reference $u\in
Q$ at $q$ is given by $R_B(u)=u\circ B:\R^m\to T_qM$. Let
$F=\hat{u}_0(B)\in H_p$ and let $f\in \mathcal{H}_p$ such that
$F=f_*$. Let $h\in \mathcal{G}$ be such that
$u=\widetilde{h}(u_0)$, we can write
\begin{align*}
R_B(u) & = u\circ B = u\circ u_0\circ F\circ u_0\\
 & = h_*\circ u_0\circ u_0^{-1}\circ F\circ u_0= h_*\circ f_*\circ
 u_0\\
 & = \widetilde{h}\circ
 \widetilde{f}(u_0).
\end{align*}
We now consider the map
$$\begin{array}{rrcl}
\Psi: & \f{g} & \to & T_{u_0}Q\\
            & \left[\xi\right] & \to & \dt
            \widetilde{\varphi_t}(u_0),
\end{array}$$
where $\varphi_t$ is the $1$-parameter group of $\xi$. $\Psi$ is
injective as $\{\varphi_t\}\subset\mathcal{G}$ and the action of
$\mathcal{G}$ on $Q$ is free. Moreover,
$$\mathrm{dim}(\f{g})=\mathrm{dim}(T_pM)+\mathrm{dim}(\f{k})=\mathrm{dim}(T_pM)+\mathrm{dim}(V_{u_0}Q)=\mathrm{dim}(T_{u_0}Q),$$
whence $\Psi$ is a linear isomorphism. Let
$\f{g}=\f{m}\oplus\f{k}$ be a reductive decomposition, we define
the horizontal subspace at $u_0$ as
$$H_{u_0}Q=\Psi(\f{m}),$$
and making use of $\mathcal{G}$ we define an horizontal
distribution on $Q$ as
$$H_uQ=\widetilde{h}_*(H_{u_0}),\qquad u=\widetilde{h}(u_0).$$
This horizontal distribution is $\mathcal{C}^{\infty}$ and
invariant by $\mathcal{G}$. In order to see that $HQ$ defines a
linear connection $\wnabla$ on $M$ we just have to show that it is
equivariant by the right action of $\bar{H}$. Let $B\in \bar{H}$,
we take $F=\hat{u}_0(B)$, and $f\in\mathcal{H}_p$ with $F=f_*$.
Let $X_u\in H_uQ$, by definition $X_u=\widetilde{h}_*(X_{u_0})$
for some $X_{u_0}\in H_{u_0}Q$ and some $h$ such that
$u=\widetilde{h}(u_0)$. This means that
$X_u=\widetilde{h}_*(\Psi([\xi]))$ for some $[\xi]\in\f{m}$. Let
$\varphi_t$ be the $1$-parameter group generated by $\xi$, we thus
have
\begin{align*}
(R_B)_*(X_u) & = (R_B)_*\circ \widetilde{h}_*\circ \Psi([\xi])\\
  & = \dt R_B\circ \widetilde{h}\circ \widetilde{\varphi_t}(u_0)=\dt \widetilde{h}\circ\widetilde{\varphi_t}\circ\widetilde{f}(u_0)\\
  & =\dt
  \widetilde{h}\circ\widetilde{f}\circ\widetilde{f}^{-1}\circ\widetilde{\varphi_t}\circ
  \widetilde{f}(u_0)\\
   & = (\widetilde{h}\circ\widetilde{f})_*\left(\dt \widetilde{f}^{-1}\circ\widetilde{\varphi_t}\circ
  \widetilde{f}(u_0)\right)\\
   & =
   (\widetilde{h}\circ\widetilde{f})_*\left(\Psi(\mathrm{Ad}_{F^{-1}}([\xi]))\right).
\end{align*}
Since $\mathrm{Ad}_{F^{-1}}([\xi])\in\f{m}$, we have
$\Psi(\mathrm{Ad}_{F^{-1}}([\xi]))\in H_{u_0}Q$, and we conclude
that $(R_B)_*(X_u)\in H_{R_B(u)}Q$ since
$R_B(u)=\widetilde{h}\circ\widetilde{f}(u_0)$.

We now study the properties of $\wnabla$. On the one hand, since
$Q$ is a reduction of $\mathcal{O}(M)$, the connection $\wnabla$
is metric, that is, $\wnabla g=0$. On the other hand, the
connection $\wnabla$ is characterized in the following way. Let
$p,q\in M$, and let $\gamma$ be a path in $M$ with $\gamma(0)=p$
and $\gamma(1)=q$. We denote by $\bar{\gamma}$ the horizontal lift
of $\gamma$ to $u_0\in Q$ with respect to $\wnabla$. The parallel
transport along $\gamma$ with respect to this connection is thus
the linear isometry $\gamma:T_pM\to T_qM$ given by $\gamma=u\circ
u_0^{-1}$, where $u=\bar{\gamma}(1)$. But since
$u=\widetilde{h}(u_0)=h_*\circ u_0$ for some $h\in\mathcal{G}$, we
have that the linear isometry $\gamma$ is exactly $h_*$. This
characterization of $\wnabla$ implies that its torsion
$\widetilde{T}$ and curvature $\widetilde{R}$ are invariant by
parallel transport, since $\wnabla$ is invariant by $\mathcal{G}$,
that is $\wnabla \widetilde{T}=0$ and $\wnabla\widetilde{R}=0$. As
a straightforward computation shows this two equations are
equivalent to
$$\wnabla R=0,\qquad \wnabla S=0.$$
\end{proof}

\begin{theorem}\label{theorem 2}
Let $(M,g)$ be a pseudo-Riemannian manifold admitting an
AS-connection $\wnabla$. Then there is a Lie pseudo-group of
isometries $\mathcal{G}$ such that $(M,g,\mathcal{G})$ is
reductive locally homogeneous.
\end{theorem}

\begin{proof}
Let $p,q\in M$, we consider a path $\gamma$ from $p$ to $q$. Since
$\wnabla$ is an AS-connection, the parallel transport
$\gamma:T_pM\to T_qM$ with respect to $\wnabla$ is a linear
isometry preserving the torsion and curvature of $\wnabla$. This
implies that there exist neighborhoods $\mathcal{U}^p$ and
$\mathcal{U}^q$, and an affine transformation
$f^{\gamma}:\mathcal{U}^p\to\mathcal{U}^q$ with respect to
$\wnabla$, such that its differential at $p$ coincides with the
parallel transport along $\gamma$ (see \cite[Vol. I, Ch. VI]{KN}).
Since $\wnabla$ is metric we have that $f^{\gamma}$ is an
isometry. We consider the set
$$\mathcal{G}=\{f^{\gamma}/\,\gamma\text{ is a path from }p\text{ to }q\}.$$
$\mathcal{G}$ is a pseudo-group of local isometries of $(M,g)$
which acts transitively on $(M,g)$, so that $(M,g)$ is locally
homogeneous. In addition, $\mathcal{G}$ coincides with the so
called transvection group of $\wnabla$, which consists of all
local affine maps of $\wnabla$ preserving its holonomy bundle
$\mathcal{P}^{\wnabla}$, that is,
$\widetilde{f}(\mathcal{P}^{\wnabla})\subset
\mathcal{P}^{\wnabla}$. This gives $\mathcal{G}$ a structure of
Lie pseudo-group. We just have to show that $(M,g,\mathcal{G})$ is
reductive. For a fixed point $p\in M$, the isotropy pseudo-group
is
$$\mathcal{H}_p=\{f^{\gamma}/\,f^{\gamma}(p)=p\}$$
which is in one to one correspondence with the set of loops based
at $p$. The linear isotropy group is thus
$$H_p=\{f^{\gamma}_{*}:T_pM\to T_pM/\, f^{\gamma}\in\mathcal{H}_p\}=\Lie{Hol}^{\wnabla}.$$
Therefore, let $(\f{g},\f{k})$ be the transitive Lie algebra
associated with $\mathcal{G}$, we have
$\f{k}\simeq\f{hol}^{\wnabla}$. We fix an orthonormal reference
$u_0$ at $p$ and consider the bundle $Q$ defined in \eqref{bundle
Q}. $Q$ is exactly the holonomy bundle of $\wnabla$ at $u_0$, and
therefore the connection $\wnabla$ reduces to $Q$ and determines a
horizontal distribution $HQ$ which is invariant by the right
action of $H_p$ and by the left action of $\mathcal{G}$ on $Q$. We
again take the linear map
$$\begin{array}{rrcl}
\Psi: & \f{g} & \to & T_{u_0}Q\\
            & \left[\xi\right] & \to & \dt
            \widetilde{\varphi_t}(u_0),
\end{array}$$
As seen before $\Psi$ is a linear isomorphism. We consider the
subspace $\f{m}=\Psi^{-1}(H_{u_0}Q)\subset\f{g}$. Obviously
$\f{g}=\f{m}\oplus\f{k}$, as $\Psi(\f{k})=V_{u_0}Q$. In addition,
let $[\xi]\in\f{m}$ with $1$-parameter group $\varphi_t$, and let
$F=f_*\in H_p$, recall that $\mathrm{Ad}_F([\xi])=[\eta]$ with
$\eta_q=\dt f\circ\varphi_t\circ f^{-1}(q)$ for every $q$ in a
certain neighborhood of $p$. Hence
$$\Psi(\mathrm{Ad}_F([\xi]))=\dt \widetilde{f}\circ\varphi_t\circ\widetilde{f}^{-1}(u_0)=\widetilde{f}_*\left((R_{F^{-1}})_*(\widetilde{\xi}_{u_0})\right).$$
Since $[\xi]\in\f{m}$ we have that $\Psi([\xi])\in H_{u_0}Q$,
whence by the invariance and the equivariance of the horizontal
distribution
$$\widetilde{f}_*\left((R_{F^{-1}})_*(\Psi([\xi]))\right)\in
\widetilde{f}_*\left(H_{R_{F^{-1}}(u_0)}Q\right)=H_{u_0}Q.$$ This
implies that $\f{m}$ is $\mathrm{Ad}(H_p)$-invariant, showing that
$(M,g)$ is reductive.
\end{proof}

\begin{remark}
A globally homogeneous pseudo-Riemannian manifold is in particular
a locally homogeneous pseudo-Riemannian manifold. Therefore the
notion of reductivity that we have defined for locally homogeneous
pseudo-Riemannian manifolds must coincide with the well known
definition of reductive homogeneous space when we consider a Lie
group $G$ as the Lie pseudo-group $\mathcal{G}$. We show below
that this is the case.
\end{remark}

Let $(M,g)$ be a globally homogeneous pseudo-Riemannian manifold
with a Lie group $G$ of (global) isometries acting transitively on
it. Let $H_p$ be the isotropy group at a point $p\in M$. We denote
by $\f{g}$ and $\f{h}$ the Lie algebras of $G$ and $H$
respectively. Recall that $(M,g,G)$ is said reductive if
$\f{g}=\f{m}\oplus\f{h}$ for some $\mathrm{Ad}(H_p)$-invariant
subspace $\f{m}\subset\f{g}$ (see for instance \cite{KN}). We
denote by $(\f{g}',\f{k}')$ the transitive Lie algebra associated
with $G$ seen as a Lie pseudo-group of local isometries, that is,
$\f{g}'$ is the set of germs of local infinitesimal
transformations of $G$. The linear isotropy group as defined in
Definition \ref{definition linear isotropy group}  is just the
image of $H_p$ under the linear isotropy representation $\lambda$
(see \cite[Ch. X]{KN}). We also recall the definition of
fundamental vector fields: let $\alpha\in\f{g}$ we define the
vector field $\alpha^*$ on $M$ as
$$\alpha^*_q=\dt L_{\exp(t\alpha)}(q),\qquad q\in M,$$
where $L_a$ denotes the left action of $a\in G$ on $M$. We
consider the following map
$$\begin{array}{rrcl}
\phi: & \f{g} & \to & \f{g}'\\
      & \alpha & \mapsto & \left[\alpha^*\right].
\end{array}$$
Note that $\phi$ is not a Lie algebra homomorphism since
$[\alpha,\beta]^*=-[\alpha^*,\beta^*]$. Nevertheless we show that
it is a linear isomorphism. Let $\alpha\in\f{g}$ be such that
$[\alpha^*]=0$, this means that $\alpha^*=0$ in a neighborhood of
$p$. In particular $\alpha^*_p=0$ and $A_{\alpha^*}|_p=0$, so that
$\alpha^*=0$. This implies $\alpha=0$, that is, $\phi$ is
injective. On the other hand, let $[\xi]\in \f{g}'$, we consider
the $1$-parameter group of $\xi$, which determines a curve
$\varphi_t\subset G$. Taking $\alpha=\dt \varphi_t$ we have
$\phi(\alpha)=[\xi]$. This proves that $\phi$ is surjective. In
addition, let $h\in H_p$ so that $h_*\in\lambda(H_p)$, the
following diagram is commutative:
\begin{equation*}
 \xymatrix{
\ar@{}[dr]|{\circlearrowleft} \f{g} \ar[d]_-{\mathrm{Ad}_h}
\ar[r]^-{\phi} & \f{g}' \ar[d]^-{\mathrm{Ad}_{h_*}}\\ \f{g}
\ar[r]^-{\phi} & \f{g}' }
\end{equation*}
In fact, let $\alpha\in\f{g}$, then
$\mathrm{Ad}_{h_*}(\alpha^*)=[\eta]$ with
$$\eta_q=\dt L_h\circ
L_{\exp{t\alpha}}\circ
L_{h^{-1}}=(L_{h})_*\left(\alpha^*_{L_{h^{-1}}(q)}\right)=(\mathrm{Ad}_{h}(\alpha))^*_q.$$
We conclude that via $\phi$ one can transform reductive
complements of $(\f{g},\f{h})$ into reductive complements of
$(\f{g}',\f{k}')$ and viceversa. This means that the notions of
reductivity from both the global and the local points of view
coincide.

\section{Invariant geometric structures}

We now consider a locally homogeneous pseudo-Riemannian manifold
$(M,g)$ endowed with a geometric structure given by a tensor field
$P$. Following \cite{Kir} we give the following definition.

\begin{definition}
An Ambrose-Singer-Kiri\v{c}enko connection (or ASK-con\-nection
for sort) on $(M,g,P)$ is a linear connection $\wnabla$ satisfying
$$\wnabla g=0,\qquad \wnabla R=0,\qquad \wnabla S=0,\qquad \wnabla P=0.$$
\end{definition}

Note that an ASK-connection is in particular an AS-connection. We
say that the geometric structure given by $P$ is invariant if the
Lie pseudo-group of isometries $\mathcal{J}$ preserving $P$, that
is
$$\mathcal{J}=\{f\in \mathcal{I},\, f^*P=P\},$$
acts transitively on $M$. The corresponding Lie equation is
$$\mathcal{L}_XP,$$
so that the infinitesimal transformations of $\mathcal{G}$ are
Killing vector fields which are infinitesimal automorphisms of the
geometric structure. A vector field $\xi$ satisfying both
$\mathcal{L}_{\xi}g=0$ and $\mathcal{L}_{\xi}P=0$ will be called a
\textit{geometric Killing vector field}. We consider the Lie
algebra $\f{j}\subset \f{i}$, which consists of germs of geometric
Killing vector fields. The Lie subalgebra $\f{j}^0\subset\f{i}^0$
is defined as the set of elements of $\f{j}$ vanishing at $p$, so
that $(\f{j},\f{j}^0)$ is a transitive Lie algebra. Let
$\f{gkill}$ be the subalgebra of $\f{kill}$ formed by all Killing
generators $(X,A)$ satisfying
$$A\cdot \nabla^jP_p+i_X\nabla^{j+1}P_p=0,\qquad j\geq 0,$$
and let $\f{gkill}^0=\f{kill}^0\cap\f{gkill}$, we have

\begin{proposition}
The transitive Lie algebra $(\f{gkill},\f{gkill}^0)$ is isomorphic
to $(\f{j},\f{j}^0)$.
\end{proposition}

\begin{proof}
Let $\xi$ be a geometric Killing vector field, let
$(X,A)=(\xi_p,A_{\xi}|_{p})$. By definition we have
$$A\cdot \nabla^j P=\mathcal{L}_{\xi}(\nabla^jP)_p-\nabla_\xi\nabla^jP_p=\mathcal{L}_{\xi}(\nabla^jP)_p-i_X\nabla^{j+1}P_p,$$
and applying Lemma \ref{lemma L de nabla} below we obtain that
$(\xi_p,A_{\xi}|_{p})\in\f{gkill}$. Making use of Lemma \ref{lemma
isomorphism with Killing gen} and Corollary \ref{corollary
realized germs} we see that the map
$$\begin{array}{rcl}
\f{j} & \to & \f{gkill}\\
\left[ \xi \right] & \mapsto & (\xi_p,A_{\xi}|_p)
\end{array}$$
is a Lie algebra isomorphism taking $\f{j}^0$ to $\f{gkill}^0$.
\end{proof}

\begin{lemma}\label{lemma L de nabla}
Let $\xi$ be a Killing vector field and $\omega$ a tensor field.
If $\mathcal{L}_{\xi}\omega=0$ then
$\mathcal{L}_{\xi}(\nabla\omega)=0$.
\end{lemma}

\begin{proof}
For the sake of simplicity we show the proof for $\omega$ a
$1$-form. The generalization for tensor fields of arbitrary type
is straightforward. By direct calculation
$$\mathcal{L}_{\xi}(\nabla\omega)(X,Y)=-\xi\cdot (\omega(\nabla_XY))+\omega\left(\nabla_{\mathcal{L}_{\xi}X}Y\right)+\omega\left(\nabla_X\mathcal{L}_{\xi}Y\right).$$
Making use of $\mathcal{L}_{\xi}\omega=0$ we obtain
$$\mathcal{L}_{\xi}(\nabla\omega)(X,Y)=\omega\left((\mathcal{L}_{\xi}\nabla)(X,Y)\right)=\omega\left(R_{\xi X}Y+\nabla^2_{XY}\xi\right).$$
But $R_{\xi}+\nabla^2\xi=0$ since it is just the affine Jacobi
equation applied to a Killing vector field $\xi$.
\end{proof}

\bigskip

We now consider a Lie pseudo-group $\mathcal{G}\subset\mathcal{J}$
acting transitively on $M$. We associate to $\mathcal{G}$ the Lie
algebra $\f{g}\subset\f{j}$ consisting on germs of local geometric
Killing vector fields with $1$-parameter group contained in
$\mathcal{G}$. The Lie algebra $\f{k}$ formed by those
$[\xi]\in\f{g}$ vanishing at $p$ is thus a Lie subalgebra of
$\f{j}^0$, and the pair $(\f{g},\f{k})$ is a transitive Lie
algebra. We take the isotropy pseudo-group $\mathcal{H}_p$ and the
linear isotropy group $H_p$ associated with $\mathcal{G}$. As
before we have that $H_p$ is a Lie subgroup of the stabilizer of
$P_p$ in $\Lie{O}(T_pM)$, and that $\f{k}\simeq \f{h}_p$. Recall
also that we have the action $\mathrm{Ad}$ of $H_p$ on $\f{g}$.

\begin{definition}
Let $(M,g,P)$ be a pseudo-Riemannian manifold endowed with a
geometric structure defined by a tensor field $P$. Let
$\mathcal{G}$ be a Lie pseudo-group of isometries acting
transitively on $(M,g,P)$ and preserving $P$. We will say that
$(M,g,P,\mathcal{G})$ is reductive if the transitive Lie algebra
$(\f{g},\f{k})$ associated with $\mathcal{G}$ can be decomposed as
$\f{g}=\f{m}\oplus \f{k}$, where $\f{m}$ is
$\mathrm{Ad}(H_p)$-invariant.
\end{definition}

\begin{theorem}
Let $(M,g,P,\mathcal{G})$ be a reductive locally homogeneous
pseudo-Rie\-mannian manifold with $P$ invariant. Then $(M,g,P)$
admits an ASK-connection.
\end{theorem}

\begin{proof}
Let $(M,g,P,\mathcal{G})$ be a reductive locally homogeneous
pseudo-Riemannian manifold with $P$ invariant, by Theorem
\ref{theorem 1} $(M,g)$ admits an AS-connection $\wnabla$. We just
have to show that $\wnabla P=0$. However, recall that $\wnabla$ is
characterized as the linear connection whose parallel transport
coincides with the differential $h_*$ of some $h\in\mathcal{G}$.
Since $\mathcal{G}$ preserves $P$, we have that $P$ is invariant
by parallel transport with respect to $\wnabla$, whence $\wnabla
P=0$.
\end{proof}

\begin{theorem}
Let $(M,g,P)$ be a pseudo-Riemannian manifold admitting an
ASK-connection $\wnabla$. Then there is a Lie pseudo-group of
isometries $\mathcal{G}$ acting transitively on $(M,g,P)$ and
preserving $P$, such that $(M,g,P,\mathcal{G})$ is reductive
locally homogeneous with $P$ invariant.
\end{theorem}

\begin{proof}
As in the proof of Theorem \ref{theorem 2} we consider the Lie
pseudo-group
$$\mathcal{G}=\{f^{\gamma}/\,\gamma\text{ is a path from }p\text{ to }q\}.$$
Since the local maps $f^{\gamma}$ are affine maps of $\wnabla$,
and $\wnabla P=0$, we have that $P$ is invariant by $\mathcal{G}$.
The same exact arguments used in the proof of Theorem \ref{theorem
2} show that $(M,g,P)$ is reductive locally homogeneous with $P$
invariant.
\end{proof}

\section{Strongly reductive locally homoge\-neous pseudo-Riemannian
manifolds}\label{section strongly reductive}

The results presented in this section apply to pseudo-Riemannian
metrics of any signature (including the Riemannian case) with or
without an extra geometric structure. In addition, these results
are new for pseudo-Riemannian metrics with or without an extra
geometric structure excluding the case of definite metrics, and in
the Riemannian case the results are new in the presence of a
geometric structure. For the already known case of Riemannian
metrics without extra geometry see \cite{Nic-Tri}. For the sake of
brevity we present here the most general case.

\bigskip

Let $(M,g)$ be a pseudo-Riemannian manifold endowed with a
geometric structure defined by a tensor field $P$. Let $p\in M$,
for every integers $r,s\geq 0$ we consider the Lie algebras
$\f{g}(p,r)$ and $\f{p}(p,s)$ given by
$$\f{g}(p,r)=\left\{A\in\f{so}(T_pM),\, A\cdot\left( \nabla^i R_p\right)=0,\hspace{2mm}i=0,\ldots,r\right\},$$
$$\f{p}(p,s)=\left\{A\in\f{so}(T_pM),\, A\cdot\left( \nabla^j P_p\right)=0,\hspace{2mm}j=0,\ldots,s\right\},$$
where $A$ acts as a derivation on the tensor algebra of $T_pM$. We
thus have filtrations
$$\f{so}(T_pM)\supset\f{g}(p,0)\supset\ldots\supset\f{g}(p,r)\supset\ldots$$
$$\f{so}(T_pM)\supset\f{p}(p,0)\supset\ldots\supset\f{p}(p,s)\supset\ldots$$
Let $k(p)$ and $l(p)$ be the first integers such that
$\f{g}(p,k(p))=\f{g}(p,k(p)+1)$ and
$\f{p}(p,l(p))=\f{p}(p,l(p)+1)$, and let
$\f{h}(p,r,s)=\f{g}(p,r)\cap\f{p}(p,s)$. We consider the complex
of filtrations
\bigskip

\noindent \resizebox{\linewidth}{!}{$\begin{matrix} \f{so}(T_pM) &
\supset & \f{g}(p,0) & \supset & \ldots & \supset &
\f{g}(p,k(p)) & = & \f{g}(p,k(p)+1)\\
\cup &  & \cup & & & & \cup & & \cup\\
\f{p}(p,0) & \supset & \f{h}(p,0,0) & \supset & \ldots & \supset &
\f{h}(p,k(p),0) & = & \f{h}(p,k(p)+1,0)\\
\cup &  & \cup & & & & \cup & & \cup\\
\vdots &  & \vdots & & & & \vdots & & \vdots\\
\cup &  & \cup & & & & \cup & & \cup\\
\f{p}(p,l(p)) & \supset & \f{h}(p,0,l(p)) & \supset & \ldots &
\supset & \f{h}(p,k(p),l(p)) & = & \f{h}(p,k(p)+1,l(p))\\
|| &  & || & & & & || & & ||\\
\f{p}(p,l(p)+1) & \supset & \f{h}(p,0,l(p)+1) & \supset & \ldots &
\supset & \f{h}(p,k(p),l(p)+1) & = & \f{h}(p,k(p)+1,l(p)+1).
\end{matrix}$}

\bigskip

\noindent To complete the notation we will denote
$\f{g}(p,-1)=\f{so}(T_pM)$, $\f{p}(p,-1)=\f{so}(T_pM)$, so that
$\f{h}(p,-1,s)=\f{p}(p,s)$ and $\f{h}(p,r,-1)=\f{g}(p,r)$.

We shall call a pair of integers $(r(p),s(p))$ in the set
$\mathbb{N}\cup\{0,-1\}$ a \textit{stabilizing pair} at $p\in M$
if $r(p)\leq k(p)$, $s(p)\leq l(p)$ and
$$\begin{matrix}
\f{h}(p,r(p),s(p)) & = & \f{h}(p,r(p)+1,s(p))\\
|| & & ||\\
\f{h}(p,r(p),s(p)+1) & = & \f{h}(p,r(p)+1,s(p)+1).
\end{matrix}$$
Note that $(k(p),l(p))$ is a stabilizing pair.

\begin{remark}\label{remark stabilizing pair}
An example of a manifold with an stabilizing pair distinct form
$(k(p),l(p))$ is exhibited in Section \ref{section reductivity
condition}.
\end{remark}

The following definition generalizes the definition of
infinitesimal homogeneous space given by Singer (see
\cite{Nic-Tri}). Consider a pair of integers
$(r,s)\in(\mathbb{N}\cup\{0,-1\})^2$. We say that $(M,g,P)$ is
\textit{$(r,s)$-infinitesimally $P$-homogeneous} if for every
$p,q\in M$ there is a linear isometry $F:T_pM\to T_q M$ such that
$$F^*(\nabla^iR_q)=\nabla^iR_p,\qquad i=0,\ldots,r+1,$$
$$F^*(\nabla^jP_q)=\nabla^jP_p,\qquad j=0,\ldots,s+1.$$
Let $p\in M$ be a fixed point and suppose that $(r(p),s(p))$ is a
stabilizing pair at $p$. If $(M,g,P)$ is
$(r(p),s(p))$-infinitesimally $P$-homogeneous, then $(r(p),s(p))$
is a stabilizing pair at all $q\in M$ (so that we can omit the
point $p$). In fact, any linear isometry $F:T_pM\to T_qM$ with
$F^*(\nabla^iR_q)=\nabla^iR_p$ and $F^*(\nabla^jP_q)=\nabla^jP_p$
for $i=0,\ldots,r(p)+1$ and $j=0,\ldots,s(p)+1$, induces
isomorphisms between $\f{h}(p,i,j)$ and $\f{h}(q,i,j)$ for $i\leq
r(p)$ and $j\leq s(p)$. Note that this means that if $(M,g,P)$ is
$(k(p),l(p))$-infinitesimally $P$-homogeneous then the numbers
$k(q)$ and $l(q)$ are independent of $q\in M$. Let $H(p,r,s)$ be
the stabilizing group of the tensors $\nabla^iR_p$, and
$\nabla^jP_p$, $0\leq i\leq r+1$, $0\leq j\leq s+1$, inside
$O(T_pM)$. It is evident that $\f{h}(p,r,s)$ is the Lie algebra of
$H(p,r,s)$.

Obviously, a locally homogeneous pseudo-Riemannian manifold with
$P$ invariant is in particular $(r,s)$-infinitesimally
$P$-homogeneous for every pair $(r,s)$. We shall see that the
converse is also true.

\begin{definition}
Let $(r,s)$ be a stabilizing pair at $p\in M$. $(M,g,P)$ is said
$(r,s)$-strongly reductive at $p$ if there is an
$\mathrm{Ad}(H(p,r,s))$-invariant subspace
$\f{n}(p,r,s)\subset\f{so}(T_pM)$ such that
$$\f{so}(T_pM)=\f{h}(p,r,s)\oplus\f{n}(p,r,s).$$
\end{definition}

\begin{lemma}\label{lemma strongly reductive independent of F}
Let $(M,g,P)$ be $(r,s)$-infinitesimally $P$-homogeneous. If
$(M,g,P)$ is $(r,s)$-strongly reductive at $p\in M$, then it is
$(r,s)$-strongly reductive at every point of $M$.
\end{lemma}

\begin{proof}
Let $q\in M$ be another point distinct from $p$, recall that
$(r,s)$ is also a stabilizing pair at $q$. Let $F:T_pM\to T_qM$ be
a linear isometry such that $F^*(\nabla^i R_q)=\nabla^iR_p$ and
$F^*(\nabla^jP_q)=\nabla^jP_p$ for $i=0,\ldots,r+1$ and
$j=0,\ldots,s+1$. $F$ induces a linear isomorphism
$\widetilde{F}:\f{so}(T_pM)\to\f{so}(T_qM)$ given by $A\mapsto
F\circ A\circ F^{-1}$. By construction it is obvious that
$\widetilde{F}(\f{h}(p,r,s))=\f{h}(q,r,s)$. Let $\f{n}(p,r,s)$ be
an $\mathrm{Ad}(H(p,r,s))$-invariant complement to $\f{h}(p,r,s)$
inside $\f{so}(T_pM)$, we define
$$\f{n}(q,r,s)=\widetilde{F}(\f{n}(p,r,s))\subset\f{so}(T_qM).$$
This subspace is independent of the isometry $F$. Indeed, let
$G:T_pM\to T_qM$ be another linear isometry with $G^*(\nabla^i
R_q)=\nabla^iR_p$ and $G^*(\nabla^jP_q)=\nabla^jP_p$ for
$i=0,\ldots,r+1$ and $j=0,\ldots,s+1$. The composition
$G^{-1}\circ F$ is an element of $O(T_pM)$. Moreover, $G^{-1}\circ
F$ stabilizes $R_p,\ldots,\nabla^{r+1}R_p$ and
$P_p,\ldots,\nabla^{s+1} P_p$, so that it is an element of
$H(p,r,s)$. Hence, for any $A\in \f{n}(p,r,s)$ we have
$$\widetilde{G}^{-1}\circ\widetilde{F}(A)=\mathrm{Ad}_{G^{-1}\circ
F}(A)\in\f{n}(p,r,s),$$ showing that $\widetilde{F}(\f{n}(p,r,s))$
does not depend on the linear isometry $F$. We finally show that
$\f{n}(q,r,s)$ is $\mathrm{Ad}(H(q,r,s))$-invariant. Let
$B\in\f{n}(q,r,s)$, there exists an element $A\in\f{n}(p,r,s)$
with $B=\widetilde{F}(A)$. Let $b\in H(q,r,s)$, we take
$a=F^{-1}\circ b\circ F\in H(p,r,s)$. Then
$$\mathrm{Ad}_{b}(B)=b\circ B\circ b^{-1}=F\circ a\circ A\circ a^{-1}\circ F^{-1}=\widetilde{F}(\mathrm{Ad}_a(A)),$$
which belongs to $\f{n}(q,r,s)$ since
$\mathrm{Ad}_a(A)\in\f{n}(p,r,s)$.
\end{proof}

\bigskip

By virtue of the previous Lemma, we say that an
$(r,s)$-infinitesimally $P$-homogeneous manifold $(M,g,P)$ is
\textit{$(r,s)$-strongly reductive} if it is $(r,s)$-strongly
reductive at some point of $M$. The same applies for locally
homogeneous spaces with $P$ invariant. The term ``strongly
reductive'' is motivated by Proposition \ref{proposition strongly
reduct implica reductiva para Isom} and Example \ref{example B3},
which show that strong reductivity implies reductivity, but the
converse is not true.

\begin{remark}
In the case $g$ is Riemannian, the Killing form of $\f{so}(T_pM)$
is definite, so that the strong reductivity condition is
automatically satisfied choosing for $\f{n}(p,r,s)$ the orthogonal
complement of $\f{h}(p,r,s)$ inside $\f{so}(T_pM)$ with respect to
the Killing form. When the presence of an extra geometric
structure is not taken into account, the integer $k(p)$
stabilizing the filtration
$$\f{so}(T_pM)\supset\f{g}(p,0)\supset\ldots\supset\f{g}(p,r)\supset\ldots$$
is known as the \textit{Singer invariant} of $(M,g)$. In this
case, the choice of $\f{g}(p,k(p))^{\bot}$ as complement of
$\f{g}(p,k(p))$ leads to the \textit{canonical AS-connec\-tion}
constructed by Kowalski in \cite{Ko2} in a similar way to the
proof of Theorem \ref{theorem 1 strongly reductive} below.
\end{remark}

Let $\pi:\mathcal{O}(M)\to M$ be the bundle of orthonormal
references with structure group $O(\nu,n-\nu)$, where $\nu$ is the
index of the metric. Let $u_0\in \mathcal{O}(M)$ with
$\pi(u_0)=p$, and $P_0=u_0^*(P_p)$. Let $\mathbf{P}$ be the space
of tensors to which $P_0$ belongs. For any pair of integers
$(r,s)\in (\mathbb{N}\cup\{0,-1\})^2$ we consider the following
$O(\nu,n-\nu)$-equivariant map:
$$\begin{array}{rrcl} \Phi_{(r,s)}: & \mathcal{O}(M) & \to &
\bigoplus_{i=0}^{k+1}\left(\bigotimes^{r+4}(\R^n)^*\right)\oplus\bigoplus_{j=0}^{s+1}\left(\bigotimes^{j}(\R^n)^*\otimes\mathbf{P}\right)\\
  & u & \mapsto &
  u^*(R_{\pi(u)},\ldots,\nabla^{r+1}R_{\pi(u)},P_{\pi(u)},\ldots,\nabla^{s+1} P_{\pi(u)}).
\end{array}$$

\begin{lemma}\label{lemma orbit pseudo}
If $(M,g,P)$ is $(r,s)$-infinitesimally $P$-homogeneous, then the
image of $\mathcal{O}(M)$ under $\Phi_{(r,s)}$ is a single
$O(\nu,n-\nu)$-orbit.
\end{lemma}

\begin{proof}
Let $u\in \mathcal{O}(M)$ and denote $\Phi=\Phi_{(r,s)}$. If
$\pi(u_0)=\pi(u)$ then $u_0$ and $u$ are in the same
$O(\nu,n-\nu)$-orbit, and since $\Phi$ is
$O(\nu,n-\nu)$-equivariant, we have that $\Phi(u_0)$ and $\Phi(u)$
are in the same $O(\nu,n-\nu)$-orbit. If $\pi(u_0)\neq\pi(u)$, let
$q=\pi(u)$, then there is a linear isometry $F:T_pM\to T_qM$ such
that $F^*(\nabla^i R_q)=\nabla^i R_p$ for $i=0,\ldots,r+1$, and
$F^*(\nabla^jP_q)=\nabla^jP_p$ for $j=0,\ldots,s+1$. $F$ induces a
map $\widetilde{F}:\mathcal{O}(M)\to \mathcal{O}(M)$ such that
$\Phi\circ \widetilde{F}=\Phi$. Since
$\pi(u)=\pi(\widetilde{F}(u_0))$, we conclude that $\Phi(u_0)$ and
$\Phi(u)$ are in the same $O(\nu,n-\nu)$-orbit.
\end{proof}

\begin{lemma}\label{lemma connection pseudo}
If $(M,g,P)$ is an $(r,s)$-infinitesimally $P$-homogeneous
manifold. Then there is a metric connection $\bar{\nabla}$ such
that $\bar{\nabla}_X(\nabla^i R)=0$ for $i=0,\ldots,r+1$, and
$\bar{\nabla}_X(\nabla^j P)=0$ for $j=0,\ldots,s+1$.
\end{lemma}

\begin{proof}
Let $u_0\in P$ with $\pi(u_0)=p$ and $\Phi=\Phi_{(r,s)}$. By Lemma
\ref{lemma orbit pseudo} the orbit $\Phi(P)$ is the homogeneous
space $O(\nu,n-\nu)/I_0$ where $I_0$ is the isotropy group of
$\Phi(u_0)$. We thus have an equivariant map
$\Phi:\mathcal{O}(M)\to O(\nu,n-\nu)/I_0$, so that
$Q=\Phi^{-1}(\Phi(u_0))$ determines a reduction of
$\mathcal{O}(M)$ with group $I_0$. Since $\Phi$ restricted to $Q$
is constant, all the tensor fields $\nabla^iR$ and $\nabla^jP$,
$i=0,\ldots,r+1$, $j=0,\ldots,s+1$, will be parallel with respect
to any connection adapted to $Q$.
\end{proof}

\begin{lemma}\label{lemma subbundles pseudo}
If $(M,g,P)$ is an $(r,s)$-infinitesimally $P$-homogeneous
manifold, then $$\f{h}(M,r,s)=\bigcup_{q\in M}\f{h}(q,r,s)$$ is a
vector subbundle of $\f{so}(M)$. If $(M,g,P)$ is moreover
$(r,s)$-strongly reductive, then $$\f{n}(M,r,s)=\bigcup_{q\in
M}\f{n}(q,r,s)$$ is a vector subbundle of $\f{so}(M)$ and
$\f{so}(M)=\f{h}(M,r,s)\oplus\f{n}(M,r,s)$.
\end{lemma}

\begin{proof}
To prove that $\f{h}(M,r,s)$ is a vector subbundle of $\f{so}(M)$
we have to find a neighborhood $U$ around every $q\in M$ admitting
local sections $\{H_1,\ldots,H_t\}$ such that
$\{H_1(y),\ldots,H_t(y)\}$ is a basis of $\f{h}(y,r,s)$ for every
$y\in U$. Let $\bar{\nabla}$ be a linear connection as in Lemma
\ref{lemma connection pseudo}, we take a normal neighborhood $U$
around $q$ with respect to $\bar{\nabla}$. Let
$\{H_1(q),\ldots,H_t(q)\}$ be a basis of $\f{h}(q,r,s)$, we extend
them by parallel transport with respect to $\bar{\nabla}$ along
radial $\bar{\nabla}$-geodesics in order to define
$\{H_1(y),\ldots,H_t(y)\}$. Since $\bar{\nabla}_X(\nabla^i R)=0$
for $i=0,\ldots,r+1$, and $\bar{\nabla}_X(\nabla^j P)=0$ for
$j=0,\ldots,s+1$, the parallel transport from $q$ to $y$ defines a
linear isometry $F:T_qM\to T_yM$ with $F^*(\nabla^i R_y)=\nabla^i
R_q$ for $i=0,\ldots,r+1$, and $F^*(\nabla^j P_y)=\nabla^j P_q$
for $j=0,\ldots,s+1$. This implies that $H_i(y)\in\f{h}(y,r,s)$.
If $(M,g,P)$ is $(r,s)$-strongly reductive, we consider the
decomposition $\f{so}(T_qM)=\f{h}(q,r,s)\oplus\f{n}(q,r,s)$ and
take a basis $\{\eta_1(q),\ldots,\eta_d(q)\}$ of $\f{n}(q,r,s)$.
Extending $\{\eta_1(q),\ldots,\eta_d(q)\}$ by parallel transport
along radial $\bar{\nabla}$-geodesics, we obtain local sections
$\eta_1,\ldots,\eta_d$ of $\f{so}(M)$ defined on $U$. As seen in
Lemma \eqref{lemma strongly reductive independent of F}, the
linear isometries $F$ determined by the parallel transport takes
$\f{n}(q,r,s)$ to $\f{n}(y,r,s)$ for $y\in U$, whence
$\{\eta_1(y),\ldots,\eta_d(y)\}$ is a basis of $\f{n}(y,r,s)$ for
every $y\in U$.
\end{proof}

\begin{theorem}\label{theorem 1 strongly reductive}
Let $(M,g,P)$ be an $(r,s)$-infinitesimally $P$-homogeneous
manifold. If $(M,g,P)$ is $(r,s)$-strongly reductive with a
decomposition $\f{so}(T_pM)=\f{n}(p,r,s)\oplus\f{h}(p,r,s)$ with
$\f{n}(M,r,s)$ $\mathrm{Ad}(H(p,r,s))$-invariant, then there is a
unique ASK-connection $\wnabla$ such that $S=\wnabla-\nabla$ is a
section of $T^*M\otimes\f{n}(M,r,s)$.
\end{theorem}

\begin{proof}
Let $\f{h}(M)$ denote $\f{h}(M,r,s)$ and let $\f{n}(M)$ denote
$\f{n}(M,r,s)$. Let $\bar{\nabla}$ be a linear connection as in
Lemma \ref{lemma connection pseudo}. We consider the tensor field
$B=\nabla-\bar{\nabla}$, which defines a section of
$T^*M\otimes\f{so}(M)$ as $\bar{\nabla}$ is metric. In virtue of
Lemma \ref{lemma subbundles pseudo} we decompose
$$B=B^{\f{h}}+B^{\f{n}},$$
with $B^{\f{h}}$ and $B^{\f{n}}$ sections of $T^*M\otimes\f{h}(M)$
and $T^*M\otimes\f{n}(M)$ respectively. We define $S=B^{\f{n}}$,
and take $\wnabla=\nabla-S$. Since $S$ is a section of
$T^*M\otimes\f{so}(M)$ we have that $\wnabla$ is metric, so that
$\wnabla g=0$. Moreover
$$\wnabla_X(\nabla^i R)=\bar{\nabla}_X(\nabla^i R)+B^{\f{h}}_X\cdot(\nabla^i R)=0,\qquad i=0,\ldots,r+1,$$
$$\wnabla_X(\nabla^j P)=\bar{\nabla}_X(\nabla^j P)+B^{\f{h}}_X\cdot(\nabla^j P)=0,\qquad j=0,\ldots,s+1,$$
since $(r,s)$ is a stabilizing pair. Finally, let $q\in M$ and
consider a normal neighborhood of $q$ with respect to $\wnabla$.
Since
$$0=\wnabla_X(\nabla^i R)=i_X(\nabla^{i+1}R)-S_X\cdot(\nabla^i R),$$
$$0=\wnabla_X(\nabla^j P)=i_X(\nabla^{j+1}P)-S_X\cdot(\nabla^j P),$$
differentiating these formulae along a radial $\wnabla$-geodesic
$\gamma(t)$ we find
$$0=0-\frac{d}{dt}\left(S_{\dot{\gamma}(t)}\cdot(\nabla^i R)_{\gamma(t)}\right)=-\left(\wnabla_{\dot{\gamma}(t)}S\right)\cdot(\nabla^i R)_{\gamma(t)},$$
$$0=0-\frac{d}{dt}\left(S_{\dot{\gamma}(t)}\cdot(\nabla^j P)_{\gamma(t)}\right)=-\left(\wnabla_{\dot{\gamma}(t)}S\right)\cdot(\nabla^j P)_{\gamma(t)},$$
for $i=0,\ldots,r$ and $j=0,\ldots,s$. This means that
$\wnabla_{\dot{\gamma}(t)}S\in\f{h}(\gamma(t),r,s)$. In addition,
as a consequence of the $\mathrm{ad}(\f{h}(M))$-invariance of
$\f{n}(M)$, the covariant derivative of a section of $\f{n}(M)$ is
again a section of $\f{n}(M)$, so that
$\wnabla_{\dot{\gamma}(t)}S\in\f{n}(\gamma(t),r,s)$. We conclude
that $\wnabla S=0$.

We finally prove uniqueness. Let $\wnabla$ and $\wnabla'$ be as in
the hypothesis, then $S-S'$ is a section of $T^*M\otimes
\f{n}(M)$. In addition $\wnabla(\nabla^iR)=\wnabla'(\nabla^i R)=0$
and $\wnabla(\nabla^j P)=\wnabla'(\nabla^j P)=0$ for all $i,j$.
These are easily obtained from the fact that the torsion and the
curvature of $\wnabla$ (resp. $\wnabla'$) are parallel with
respect to $\wnabla$ (resp. $\wnabla'$), and from $\wnabla
P=\wnabla' P=0$. This implies that $S-S'$ is a section of
$T^*M\otimes \f{h}(M)$, and then $S=S'$ and $\wnabla=\wnabla'$.
\end{proof}

\begin{corollary}
Let $(r,s)$ and $(r',s')$ be stabilizing pairs. If
$\f{n}(p,r,s)\subset\f{n}(p,r',s')$, then the connections
$\wnabla$ and $\wnabla'$ constructed from them coincide.
\end{corollary}

\begin{proof}
This is evident since $S=\wnabla-\nabla$ is a section of both
$\f{n}(M,r,s)$ and $\f{n}(M,r',s')$.
\end{proof}

\bigskip



As we have seen, a strongly reductive locally homogeneous
pseudo-Rieman\-nian manifold $(M,g,P)$ with $P$ invariant admits
an ASK-connec\-tion, so by Theorem \ref{theorem 2} there is a Lie
pseudo-group $\mathcal{G}$ (which is not necessarily the full
isometry pseudo-group) acting transitively by isometries and
preserving $P$ such that $(M,g,P,\mathcal{G})$ is reductive.
Moreover, we shall show that strongly reductive locally
homogeneous spaces with an invariant geometric structure $P$ are
reductive for the action of the full pseudo-group of isometries
preserving $P$. In order to prove that we will make use of some
results contained in Section \ref{section reconstruction} and the
following Lemma.

\begin{lemma}\label{lemma h0 contained in gkill}
Let $\wnabla$ be an ASK-connection with curvature $K$ and torsion
$T$. Let $p\in M$, and let $A\in\f{so}(T_pM)$ be such that $A\cdot
K_p=0$, $A\cdot T_p=0$ and $A\cdot P_p=0$. Then $A\cdot
\nabla^iR_p=0$ and $A\cdot\nabla^jP_p=0$ for all $i,j\geq 0$.
\end{lemma}

\begin{proof}
The curvature and torsion of $\wnabla$ are related to $R$ and $S$
by
$$T_XY=S_YX-S_XY, \qquad K_{XY}=R_{XY}+[S_X,S_Y]+S_{T_XY}.$$
Making use of these formulae in conjunction with $\wnabla R=0$ and
$\wnabla S=0$, an inductive argument gives that
$\wnabla(\nabla^iR)=0$ for all $i\geq 0$. A similar computation
gives $\wnabla(\nabla^jP)=0$ for all $j\geq 0$. This means that
$$i_X\nabla^{i+1}R=S_X\cdot \nabla^iR,\qquad i_X\nabla^{j+1}P=S_X\cdot \nabla^jP,$$
for all $i,j\geq 0$. Let now $A\in\f{so}(T_pM)$ be such that
$A\cdot K_p=0$, $A\cdot T_p=0$ and $A\cdot P_p=0$. By Corollary
\ref{corollary h deriving S} $A\cdot S_p=0$, hence $A\cdot R_p=0$.
A simple computation making use of the previous formulae leads to
$$(A\cdot \nabla^{i+1}R_p)_X=(A\cdot S_p)_X\cdot\nabla^iR_p+(S_p)_X\cdot(A\cdot \nabla^iR_p), \qquad i\geq 0,$$
$$(A\cdot \nabla^{j+1}P_p)_X=(A\cdot S_p)_X\cdot\nabla^{j}P_p+(S_p)_X\cdot(A\cdot \nabla^jP_p), \qquad j\geq 0.$$
Therefore, by induction on $i$ and $j$ we obtain that $A\cdot
\nabla^iR_p=0$ and $A\cdot\nabla^jP_p=0$ for all $i,j\geq 0$.
\end{proof}

\begin{proposition}\label{proposition strongly reduct implica reductiva para Isom}
If $(M,g,P)$ is $(r,s)$-strongly reductive, then
$(M,g,\mathcal{J})$ is reductive, where $\mathcal{J}$ is the full
Lie pseudo-group of local isometries preserving $P$.
\end{proposition}

\begin{proof}
Let $\f{so}(T_pM)=\f{n}(p,r,s)\oplus\f{h}(p,r,s)$, and let
$\wnabla$ be the associated ASK-connection. Let $K$ and $T$ be the
curvature and the torsion tensor fields of $\wnabla$ respectively.
The triple $(K,T,P)$ defines an infinitesimal model (see
Definition \ref{definition infinitesimal model} and Proposition
\ref{proposition infinitesimal model}), and we can consider the
associated Nomizu construction, that is, we define the Lie algebra
$\f{g}_0=T_pM\oplus\f{h}_0$ with the usual brackets, where
$$\f{h}_0=\{A\in\f{so}(T_pM)/\, A\cdot K_p=0,\, A\cdot T_p=0,\, A\cdot P_p=0\}.$$
By Proposition \ref{proposition h=k} the Lie algebra $\f{h}_0$ is
equal to $\f{h}(p,r,s)$. On the other hand,
$\f{h}_0\subset\f{gkill}^0\simeq \f{j}^0$ by Lemma \ref{lemma h0
contained in gkill}, and $\f{gkill}^0\subset\f{h}$ by definition,
whence $\f{gkill}^0\subset\f{h}=\f{h}_0$. We thus define the
following Lie algebra isomorphism
$$\begin{array}{rrcl}
\Phi: & \f{g}_0 & \to & \f{gkill}\\
      & X+A  & \mapsto & (X,(S_0)_X+A).
\end{array}$$
The image of $T_pM$ defines a complement $\f{m}$ of $\f{gkill}^0$.
Making use of Lemma \ref{lemma H deriving S} we have that
$\mathrm{Ad}_{B}(S_X)=S_{BX}$ for all $B$ in $H(p,r,s)$ and all
$X\in T_pM$. Since the linear isotropy group $H_p$  is contained
in $H(p,r,s)$ we have that $\f{m}$ is
$\mathrm{Ad}(H_p)$-invariant.
\end{proof}

\section{Reconstruction of strongly reductive\\ locally homogeneous
spaces}\label{section reconstruction}

We first show a uniqueness result satisfied by strongly reductive
locally homogeneous pseudo-Riemannian manifolds with an invariant
geometric structure.

\begin{proposition}\label{proposition uniqueness pseudo}
Let $(M,g,P)$ and $(M',g',P')$ be pseudo-Riemannian manifolds with
tensor fields $P$ and $P'$. Suppose $(M',g',P')$ is locally
homogeneous with $P'$ invariant. Suppose furthermore that
$(M',g',P')$ is $(r,s)$-strongly reductive for some stabilizing
pair $(r,s)$. If for each point $p\in M$ there is a linear
isometry $F:T_p M\to T_oM'$ (where $o\in M'$ can be fixed) such
that $F^*(\nabla'^i R'_o)=\nabla^i R_p$ for $=0,\ldots,r+1$, and
$F^*(\nabla'^jP_o)=\nabla^jP_p$ for $j=0,\ldots,s+1$. Then
$(M,g,P)$ is locally homogeneous with $P$ invariant and locally
isometric to $(M',g',P')$ preserving $P$ and $P'$.
\end{proposition}

\begin{proof}
Note first of all that $(M,g,P)$ is $(r,s)$-infinitesimally
$P$-homoge\-neous and $(r,s)$-strongly reductive, so that
$(M,g,P)$ is locally homogeneous with $P$ invariant. Let $\wnabla$
and $\wnabla'$ be connections on $M$ and $M'$ respectively as in
Theorem \ref{theorem 1 strongly reductive}. Let $S=\nabla-\wnabla$
and $S'=\nabla'-\wnabla'$, and let $F:T_pM\to T_oM'$ be as in the
hypothesis. It is obvious that $F^*(S'_o)-S_p\in
T^*_pM\otimes\f{n}(p,r,s)$. In addition
$$\left(F^*(S'_{o})_X-(S_{p})_X\right)\cdot(\nabla^{i}R_p)=i_X\nabla^{i+1}R_p-i_X\nabla^{i+1}R_p=0,\hspace{1em} i=0,\ldots,r,$$
$$\left(F^*(S'_{o})_X-(S_{p})_X\right)\cdot(\nabla^{j}P_p)=i_X\nabla^{j+1}P_p-i_X\nabla^{j+1}P_p=0,\hspace{1em} j=0,\ldots,s,$$
so that $F^*(S'_{o})_X-(S_{p})_X\in\f{h}(p,r,s)$. We conclude that
$F^*(S'_o)=S_p$. Since the torsion of $\wnabla$ is $S_YX-S_XY$,
and a similar formula holds for the torsion of $\wnabla'$, as a
simple inspection shows, $F$ preserves the curvature and the
torsion of $\wnabla$ and $\wnabla'$, which are parallel with
respect to $\wnabla$. Therefore, there are neighborhoods $U$ and
$V$ around $p$ and $o$ respectively, and an affine map $f:U\to V$
with respect to $\wnabla$ and $\wnabla'$ (see \cite[Ch. 7]{KN}).
Since $\wnabla$ and $\wnabla'$ are metric and $\wnabla
P=\wnabla'P'=0$, we have that $f$ is an isometry preserving $P$
and $P'$.
\end{proof}

\bigskip

Theorem \ref{theorem 1 strongly reductive} and Proposition
\ref{proposition uniqueness pseudo} suggest the possibility of
reconstructing a strongly reductive locally homogeneous manifold
$(M,g,P)$ with $P$ invariant from the knowledge of the curvature
tensor field, the tensor field $P$, and their covariant
derivatives at a point $p\in M$ up to finite order. In order to
prove this result we must first examine the algebraic properties
of the curvature tensor field, $P$ and its covariant derivatives.

Let $(M,g,P)$ be a locally homogeneous pseudo-Riemannian manifold
with $P$ invariant. We fix a point $p\in M$ and set $V=T_pM$.
Consider the tensors $R^i=\nabla^i R_p$ and $P^j=\nabla^j P_p$ for
$i,j\geq 0$. One has
\begin{align}
& R^0_{XYZW}=-R^0_{YXZW}=R^0_{ZWXY},\label{infinit information
properties 1}\\
& \Cyclic_{\SXYZ} R^0_{XYZW}=0,\\
& R^1_{X Y ZVW} = -R^1_{X ZY VW} = R^1_{X VWY Z},\\
& \Cyclic_{\SYZV} R^1_{XYZVW}=0,\\
& \Cyclic_{\SXYZ} R^1_{XYZVW}=0,\\
& R^{i+2}_{YX}-R^{i+2}_{XY}=R^0_{XY}\cdot R^i,\label{infinit
information properties 6}\\
& P^{j+2}_{YX}-P^{j+2}_{XY}=R^0_{XY}\cdot P^j,\label{infinit
information properties 7}
\end{align}
for $i,j\geq 0$, where $R^0_{XY}$ is acting as a derivation on the
tensor algebra. In addition, let $\wnabla$ be an ASK-connection
and $S=\nabla-\wnabla$, we have that
$$i_XR^{i+1}=S_X\cdot R^i,\qquad i_XP^{j+1}=S_X\cdot P^j,$$
for $0\leq i\leq r+1$, $0\leq j\leq s+1$, where $(r,s)$ is a
stabilizing pair. We thus consider the following linear maps
$$\begin{array}{rrcl}
\mu_{i,j}: & \f{so}(V) & \to & W_{i,j}\\
           &    A      & \mapsto & (A\cdot R^0,\ldots, A\cdot R^i,A\cdot P^0,\ldots,A\cdot
           P^j),
\end{array}$$
and
$$\begin{array}{rrcl}
\nu: & V & \to & W_{r+1,s+1}\\
     & X & \mapsto &
     (i_XR^1,\ldots,i_XR^{r+2},i_XP^1,\ldots,i_XP^{s+2}),
\end{array}$$
with
$$W_{i,j}=\left[\bigoplus_{\alpha=0}^{i}\left(\otimes^{\alpha+4}V^*\right)\right]\otimes\left[\bigoplus_{\beta=0}^j\left((\otimes^{\beta}V^*)\otimes\mathbf{P}\right)\right],$$
where $\mathbf{P}$ is the space of tensors to which $P^0$ belongs.
The previous discussion for a stabilizing pair $(r,s)$ thus gives
\begin{equation}\label{property nu}
\nu(V)\subset \mu_{r+1,s+1}(\f{so}(V)),
\end{equation}
and
\begin{equation}\label{property ker mu}
ker(\mu_{r,s})=ker(\mu_{r+1,s})=ker(\mu_{r,s+1})=ker(\mu_{r+1,s+1}).
\end{equation}
Finally, let $H(r,s)$ be the stabilizer of $R^0,\ldots,R^{r+1}$
and $P^0,\ldots,P^{s+1}$ inside $O(V)$. In view of Theorem
\ref{theorem 1 strongly reductive}, to assure the existence of an
ASK-connection we need that
\begin{equation}\label{condition strongly-reductivity}
\f{so}(V)=ker(\mu_{r,s})\oplus\f{n}
\end{equation}
for an $\mathrm{Ad}_{H(r,s)}$-invariant subspace $\f{n}$. We shall
prove the following result.

\begin{theorem}\label{reconstruction theorem pseudo-Riemannian}
Let $V$ be a vector space endowed with an inner product $\langle\,
,\, \rangle$. Let $R^0$,...,$R^{r+2}$, $P^0,\ldots,P^{s+2}$ be
tensors on $V$ satisfying \eqref{infinit information properties
1},...,\eqref{infinit information properties 7} for $0\leq i\leq
r$ and $0\leq j\leq s$, and such that \eqref{property nu},
\eqref{property ker mu}, and \eqref{condition
strongly-reductivity} hold. Then
\begin{enumerate}
\item There is an $(r,s)$-strongly reductive locally homogeneous
pseudo-Rie\-mannian manifold $(M,g,P)$ with $P$ invariant, whose
curvature tensor field, $P$, and their covariant derivatives
coincide with $R^0,\ldots,\\ R^{r+2}$, $P^0,\ldots,P^{s+2}$ at a
point $p\in M$. Moreover, $(M,g,P)$ is unique up to local isometry
preserving $P$.

\item If the infinitesimal data $R^0,\ldots,R^{r+2}$,
$P^0,\ldots,P^{s+2}$ is regular (see Definitions \ref{definition
regular 1} and \ref{definition regular 2}), then there is an
$(r,s)$-strongly reductive globally homogeneous pseudo-Riemannian
space $(G_0/H_0,g,P)$, whose curvature tensor field, $P$, and
their covariant derivatives coincide with $R^0,\ldots,R^{r+2}$,
$P^0,\ldots,P^{s+2}$ at a point $p\in M$. $(G_0/H_0,g,P)$ is
moreover unique up to local isometry preserving $P$.
\end{enumerate}
\end{theorem}

\begin{corollary}
An $(r,s)$-strongly reductive locally homogeneous
pseudo-Rieman\-nian manifold $(M,g,P)$ with $P$ invariant can be
reconstructed (up to local isometry) from the data $R_p,\ldots,
\nabla^{r+2}R_p$, $P_p,\ldots,\nabla^{s+2}P_p$, where $(r,s)$ is a
stabilizing pair.
\end{corollary}

Before proving Theorem \ref{reconstruction theorem
pseudo-Riemannian} we need to recall the definition of
\textit{infinitesimal model} and show that an infinitesimal model
can be associated to every suitable infinitesimal data
$R^0,\ldots,R^{s+2}$, $P^0,\ldots,P^{r+2}$ satisfying the
hypotheses of Theorem \ref{reconstruction theorem
pseudo-Riemannian}.

Let $V$ be a vector space with an inner product $\langle\,
,\,\rangle$, and let $P$ be a tensor on $V$. We consider morphisms
$$T:V\to\mathrm{End}(V),\qquad K:V\times V\to \mathrm{End}(V).$$

\begin{definition}\label{definition infinitesimal model}
A triple $(T,K,P)$ is called an infinitesimal model if the
following properties are satisfied:
\begin{align}
& T_XY +T_YX=0 \label{infinitesimal model properties 1}\\
& K_{XY}Z+K_{YX}Z=0\label{infinitesimal model properties 2}\\
& \langle K_{XY}Z,W\rangle +\langle K_{WZ}X,Y\rangle =0\label{infinitesimal model properties 3}\\
& K_{XY}\cdot T=0\label{infinitesimal model properties T}\\
& K_{XY}\cdot K=0\label{infinitesimal model properties R}\\
& K_{XY}\cdot P=0\label{infinitesimal model properties P}\\
& \Cyclic_{\SXYZ}(K_{XY}Z+T_{T_XY}Z)=0\label{infinitesimal model properties 7}\\
& \Cyclic_{\SXYZ} K_{T_XYZ}=0 \label{infinitesimal model
properties 8}
\end{align}
\end{definition}

When the geometric structure $P$ is absent, an infinitesimal model
is just a pair $(T,K)$ satisfying the previous properties with the
exception of \eqref{infinitesimal model properties P}.

Let $\wnabla$ be an ASK-connection on $(M,g,P)$. For a fixed point
$p\in M$ we take $V=T_pM$, $\langle\, ,\, \rangle=g_p$, $P=P_p$,
and $T$ and $K$ the torsion and curvature of $\wnabla$ at $p$
respectively. It is easy to see that in that case $(T,K,P)$
satisfies \eqref{infinitesimal model properties
1},...,\eqref{infinitesimal model properties 8}, so that it
defines an infinitesimal model. The converse is true under
suitable conditions that must be explained. To every infinitesimal
model $(T,K,P)$ one can associate the so called \textit{Nomizu
construction}, that is, the Lie algebra
$$\f{g}_0=\f{h}_0\oplus V,$$
where
$$\f{h}_0=\{A\in\f{so}(V)/\, A\cdot K=0,\,A\cdot T=0,\, A\cdot P=0\},$$
and the Lie brackets are defined by
\begin{align*}
[A,B]&=AB-BA,\qquad A,B\in\f{h}_0,\\
[A,X]&=A\cdot X,\qquad A\in\f{h}_0,X\in V,\\
[X,Y]&=-T_{X}Y+K_{XY},\qquad X,Y\in V.
\end{align*}

Note that $K_{XY}\in\f{h}_0$. Let $\f{h}_0'$ be the subalgebra
spanned by all elements $K_{XY}$ (which in the case when $(T,K,P)$
comes from an ASK-connection coincides with the holonomy algebra
of $\wnabla$), the Lie algebra $\f{g}_0'=\f{h}_0'\oplus V$ is the
so called \textit{transvection algebra} (see \cite{Ko}).

We now consider the abstract simply-connected Lie group $G_0$ with
Lie algebra $\f{g}_0$, and its connected Lie subgroup $H_0$ with
Lie algebra $\f{h}_0$. We also consider the simply-connected Lie
group $G'_0$ with Lie algebra $\f{g}'_0$, and its connected Lie
subgroup $H'_0$ with Lie algebra $\f{h}'_0$.

\begin{definition}\label{definition regular 1}
We say that the infinitesimal model $(T,K,P)$ is regular if $H_0$
is closed in $G_0$. On the other hand, we say that the
transvection algebra $(\f{g}'_0,\f{h}'_0)$ is regular if $H'_0$ is
closed in $G'_0$.
\end{definition}


In the case when $(T,K,P)$ (resp. the transvection algebra) is
regular, the quotient $G_0/H_0$ (resp. $G'_0/H'_0$) is a
pseudo-Riemannian homogeneous space with an invariant tensor field
$\bar{P}$ coinciding with $P$ at the origin. 

We now show how to associate an infinitesimal model to every
suitable data $R^0,\ldots,R^{r+2}$, $P^0,\ldots,P^{s+2}$ on $V$
satisfying the hypotheses of Theorem \ref{reconstruction theorem
pseudo-Riemannian}. We define $\f{h}=ker(\mu_{r+1,s+1})$, and
consider an $Ad(H(r,s))$-invariant complement $\f{n}$ of $\f{h}$
inside $\f{so}(V)$. From \eqref{property nu} we have that for
every $X\in V$ there is an endomorphism $A(X)\in \f{so}(V)$ such
that
\begin{align*}
i_XR^{i+1} & = A(X)\cdot R^{i},\qquad 0\leq i\leq r+1,\\
i_XP^{j+1} & = A(X)\cdot P^j,\qquad 0\leq j\leq s+1.
\end{align*}
We decompose $A(X)=A_1(X)+A_2(X)$, where $A_1(X)\in\f{h}$ and
$A_2(X)\in\f{n}$. Note that $A(X)$ is uniquely determined up to an
$\f{h}$-component, so that we can take the uniquely defined map
$$\begin{array}{rrcl}
S: & V & \to & \f{n}\\
   & X & \mapsto & S_X=A_2(X).
\end{array}$$
By the definition of $\f{h}$ it is evident that
\begin{align}
i_XR^{i+1} & = S_X\cdot R^{i},\qquad 0\leq i\leq r+1,\label{property SXR}\\
i_XP^{j+1} & = S_X\cdot P^j,\qquad 0\leq j\leq s+1.\label{property
SXP}
\end{align}
Moreover, by the same arguments used in \cite{Nic-Tri} one sees
that $S$ is a linear map.

\begin{lemma}\label{lemma H deriving S}
Let $B\in H(r,s)$, then $\mathrm{Ad}_B(S_{X})=S_{BX}$ for every
$X\in V$.
\end{lemma}

\begin{proof}
By the definition of $H(r,s)$ and \eqref{property SXR} and
\eqref{property SXP} we have for $0\leq i\leq r$ and $0\leq j\leq
s$
\begin{align*}
R^{i+1}_{XZ_1\ldots Z_{i+4}} & = (B\cdot R^{i+1})_{XZ_1\ldots Z_{i+4}}\\
& = R^{i+1}_{B^{-1}XB^{-1}Z_1\ldots B^{-1}Z_{i+4}}\\
& =\left(S_{B^{-1}X}\cdot R^i\right)_{B^{-1}Z_1\ldots
B^{-1}Z_{i+4}}\\
& = -\sum_{\alpha}R^i_{B^{-1}Z_1\ldots
S_{B^{-1}X}B^{-1}Z_{\alpha} \ldots B^{-1}Z_{i+4}}\\
& = -\sum_{\alpha}R^i_{B^{-1}Z_1\ldots
B^{-1}BS_{B^{-1}X}B^{-1}Z_{\alpha} \ldots B^{-1}Z_{i+4}}\\
& = -\sum_{\alpha}(B\cdot R^i)_{Z_1\ldots
\mathrm{Ad}_B(S_{B^{-1}X})Z_{\alpha} \ldots Z_{i+4}}\\
& =-\sum_{\alpha} R^i_{Z_1\ldots
\mathrm{Ad}_B(S_{B^{-1}X})Z_{\alpha} \ldots Z_{i+4}}\\
& =\left(\mathrm{Ad}_B(S_{B^{-1}X})\cdot R^i\right)_{Z_1\ldots
Z_{i+4}}.
\end{align*}
On the other hand $i_{X}R^{i+1}=S_X\cdot R^i$, so that
$\mathrm{Ad}_B(S_{B^{-1}X})\cdot R^i-S_X$ belongs to $\f{h}$.
Since $S_X$ belongs to $\f{n}$ which is
$\mathrm{Ad}(H(r,s))$-invariant, we also have that
$\mathrm{Ad}_B(S_{B^{-1}X})\cdot R^i-S_X$ belongs to $\f{n}$. This
implies that $\mathrm{Ad}_B(S_{B^{-1}X})\cdot R^i-S_X=0$.
\end{proof}

\begin{corollary}\label{corollary h deriving S}
Let $A\in \f{h}$, then $A\cdot S=0$.
\end{corollary}

We take
\begin{align*}
T_XY &=S_YX-S_XY,\\
K_{XY} &= R^0_{XY}+[S_X,S_Y]+S_{T_XY},\\
P &= P^0.
\end{align*}

\begin{proposition}\label{proposition infinitesimal model}
The triple $(T,K,P)$ is an infinitesimal model.
\end{proposition}

\begin{proof}
We have to show that $(T,K,P)$ satisfies \eqref{infinitesimal
model properties 1},...,\eqref{infinitesimal model properties 8}.
For \eqref{infinitesimal model properties 1}, \eqref{infinitesimal
model properties 2}, \eqref{infinitesimal model properties 3},
\eqref{infinitesimal model properties 7} and \eqref{infinitesimal
model properties 8} one uses exactly the same arguments used in
\cite{Nic-Tri}. For the remaining, we observe that
\begin{align*}
R^{i+2}_{XY}-R^{i+2}_{YX}=\left([S_X,S_Y]+S_{T_XY}\right)\cdot
R^i,\qquad 0\leq i\leq r,\\
P^{j+2}_{XY}-P^{j+2}_{YX}=\left([S_X,S_Y]+S_{T_XY}\right)\cdot
P^j,\qquad 0\leq j\leq s.
\end{align*}
In fact, by \eqref{property SXR}
\begin{align*}
R^{i+2}_{XYZ_1\ldots Z_{i+4}} & = (i_XR^{i+2})_{YZ_1\ldots
Z_i}=(S_X\cdot R^{i+1})_{XYZ_1\ldots Z_{i+4}}\\
 & = -R^{i+1}_{S_XYZ_1\ldots Z_{i+4}}-\sum_{\alpha=1}^{i+4}R^{i+1}_{YZ_1\ldots S_XZ_{\alpha}\ldots
 Z_{i+4}}\\
 & = -\left(i_{S_XY}R^{i+1}\right)_{Z_1\ldots Z_{i+4}}-\sum_{\alpha=1}^{i+4}\left(i_YR^{i+1}\right)_{Z_1\ldots S_XZ_{\alpha}\ldots
 Z_{i+4}}\\
 & = -\left(S_{S_XY}\cdot R^i\right)_{Z_1\ldots Z_{i+4}}-\sum_{\alpha=1}^{i+4}\left(S_Y\cdot R^i\right)_{Z_1\ldots S_XZ_{\alpha}\ldots
 Z_{i+4}}\\
 & = \sum_{\alpha=1}^{i+4} R^i_{Z_1\ldots S_{S_XY}Z_{\alpha}\ldots
 Z_{i+4}}+\sum_{\alpha,\beta=1}^{i+4} R^i_{Z_1\ldots S_XZ_{\alpha}\ldots S_YZ_{\beta}\ldots
 Z_{i+4}},
\end{align*}
and by \eqref{property SXP} a similar argument holds for
$P^{j+2}_{XY}$. Skew-symmetrizing in $X,Y$ we obtain the desired
formulae. Therefore, by \eqref{infinit information properties 6}
and \eqref{infinit information properties 7} and the definition of
$K$ we obtain that $K\cdot R^i=0$ and $K_{XY}\cdot P^j=0$, for
$0\leq i\leq r$ and $0\leq j\leq s$, so in particular $K_{XY}\cdot
P^0=0$ and $K_{XY}\cdot R^0$. Making use of \eqref{property ker
mu} this implies that $K_{XY}\in \f{h}$, whence $K_{XY}\cdot S=0$
by Corollary \ref{corollary h deriving S}, giving that
$K_{XY}\cdot T=0$. Finally, as a straightforward computation
shows, for $A\in\f{h}$
\begin{equation}\label{h deriving K}
(A\cdot K)_{XY} = (A\cdot R^0)_{XY}+[(A\cdot S)_X,S_Y]-[(A\cdot
S)_Y,S_X]+S_{(A\cdot T)_XY},
\end{equation}
so that $K_{XY}\cdot K=0$.
\end{proof}

\begin{proposition}\label{proposition h=k}
$$\f{h}=\f{h}_0=\{A\in\f{so}(V)/\, A\cdot K=0,\,A\cdot T=0,\, A\cdot P=0\}.$$
\end{proposition}

\begin{proof}
Let $A\in \f{h}$, by Corollary \ref{corollary h deriving S} we
have $A\cdot S=0$, which implies $A\cdot T=0$. In addition, by
\eqref{h deriving K} we have $A\cdot K=0$. Since $P=P^0$, by
definition we deduce that $A\in \f{h}_0$, hence
$\f{h}\subset\f{h}_0$. Conversely, let $A\in \f{h}_0$. We have
that $A\cdot S=0$ since $S$ is recovered from $T$ making use of
$$2\langle S_XY,Z\rangle=-\langle T_XY,Z\rangle+\langle T_YZ,X\rangle-\langle T_ZX,Y\rangle.$$
On the other hand, by \eqref{h deriving K} we obtain $A\cdot
R^0=0$, and since $P=P^0$ we also have $A\cdot P^0=0$. Now, a
simple computation (see Lemma \ref{lemma H deriving S}) shows that
\begin{align*}(A\cdot R^{i+1})_X &=[A,S_X]\cdot R^i-S_{AX}\cdot
R^i+S_X\cdot(A\cdot R^i)\\
&= (A\cdot S)_X\cdot R^i+S_X\cdot(A\cdot R^i),\qquad 0\leq i\leq r+1,\\
(A\cdot P^{j+1})_X &=[A,S_X]\cdot P^j-S_{AX}\cdot
P^j+S_X\cdot(A\cdot P^j)\\
& = (A\cdot S)_X\cdot P^j+S_X\cdot(A\cdot P^j),\qquad 0\leq j\leq
s+1.
\end{align*}
Using these formulae, by an inductive argument on the indices $i$
and $j$ we obtain that $A\cdot R^i=0$ and $A\cdot P^j=0$ for
$0\leq i\leq r+1$ and $0\leq j\leq s+1$. Hence $A\in\f{h}$,
proving that $\f{h}_0\subset \f{h}$.
\end{proof}

\begin{definition}\label{definition regular 2}
The infinitesimal data $R^0,\ldots,R^{r+2}$, $P^0,\ldots,P^{s+2}$
is said regular if the associated infinitesimal model $(T,K,P)$ is
regular.
\end{definition}

\begin{remark}\label{remark recovery of infinitesimal information}
$R^0,\ldots,R^{r+2}$, $P^0,\ldots,P^{s+2}$ is recovered from the
infinitesimal model $(T,K,P)$ in the following way. As we have
seen $S$ is obtained from $T$ by
$$2\langle S_XY,Z\rangle=-\langle T_XY,Z\rangle+\langle T_YZ,X\rangle-\langle T_ZX,Y\rangle.$$
With $T$ and $S$ one recovers $R^0$ using the definition of $K$.
Finally, knowing $R^0$ and $P^0=P$, and using \eqref{property SXR}
and \eqref{property SXP}, one can subsequently obtain $R^i$ and
$P^j$.
\end{remark}

We are now in position to prove Theorem \ref{reconstruction
theorem pseudo-Riemannian}.

\bigskip

\begin{proof}[Proof of Theorem \ref{reconstruction theorem pseudo-Riemannian}]
Suppose that the infinitesimal model $(T,K,P)$ associated with the
infinitesimal data $R^0,\ldots,R^{r+2}$, $P^0,\ldots,P^{s+2}$ is
regular. We consider the Nomizu construction
$\f{g}_0=\f{h}_0\oplus V$, and the Lie groups $G_0$ and $H_0$,
where $G_0$ is the simply-connected Lie group with Lie algebra
$\f{g}_0$ and $H_0$ is its connected Lie subgroup with Lie algebra
$\f{h}_0$. Since $H_0$ is closed in $G_0$ we consider the
homogeneous space $G_0/H_0$, which is a reductive homogeneous
space with reductive decomposition $\f{g}_0=\f{h}_0\oplus V$.
Identifying $V=T_oG_0/H_0$ we extend $\langle\,,\,\rangle$ and $P$
to a $G_0$-invariant Riemannian metric $g$ and a $G_0$-invariant
tensor field $\bar{P}$ on $G_0/H_0$ respectively. We consider the
canonical connection associated with that reductive decomposition
(see \cite[Ch. X]{KN}), which is an ASK-connection whose curvature
and torsion coincide with $K$ and $T$. As a straightforward
computation using the properties of the canonical connection
shows, $R^0,\ldots,R^{r+2}$, $P^0,\ldots,P^{s+2}$ coincide with
the covariant derivatives of the curvature of $g$ and $\bar{P}$ at
the origin $o\in G_0/H_0$. By the identification of $T_oG_0/H_0$
with $V$, we have that $G_0/H_0$ is $(r,s)$-strongly reductive.
This proves the second part of the theorem.

Concerning the first part, we adapt the arguments used in
\cite{Tri-Wat}. Let $(T,K,P)$ be the infinitesimal model
associated with the infinitesimal data $R^0,\ldots,R^{r+2}$,
$P^0,\ldots,P^{s+2}$, which now need not be regular. We consider
the corresponding Nomizu construction $\f{g}_0=\f{h}_0\oplus V$.
Let $G_0$ be the simply-connected Lie group with Lie algebra
$\f{g}_0$, we choose an orthonormal basis $\{e_1,\ldots,e_n\}$ of
$V$, and denote by $\{e^1,\ldots,e^n\}$ its dual basis. Let
$\{A_1,\ldots,A_d\}$ be a basis of $\f{h}_0$, and
$\{A^1,\ldots,A^d\}$ its dual basis. We write
\begin{align*}
T &= T_{\alpha\beta}^{\gamma}e^{\alpha}\otimes e^{\beta}\otimes
e_{\gamma},\\
K & = K_{\alpha\beta\gamma}^{\delta}e^{\alpha}\otimes
e^{\beta}\otimes e^{\gamma}\otimes e_{\delta},\\
P
&=P_{\alpha_1\ldots\alpha_u}^{\beta_1\ldots\beta_v}e^{\alpha_1}\otimes\ldots\otimes
e^{\alpha_u}\otimes e_{\beta_1}\otimes\ldots\otimes e_{\beta_v},\\
\end{align*}
and define
$$\omega^{\alpha}_{\beta}=e^{\alpha}(A_{\gamma}(e_{\beta}))\otimes A^{\gamma},$$
where Einstein's summation convention is used. Note that
$\omega^{\alpha}_{\beta}\in \f{g}^*$, so that
$$\omega=\omega^{\alpha}_{\beta}A_{\alpha}\otimes A^{\beta}$$
defines a left invariant $2$-form on $G_0$ with values in
$\f{h}_0\subset\f{so}(V)$. Making use of the brackets defined in
$\f{g}_0$ we easily obtain
\begin{align}
de^{\alpha}
&=\frac{1}{2}T_{\beta\gamma}^{\alpha}-\omega^{\alpha}_{\beta}\wedge
e^{\beta},\label{structure equation 1}\\
d\omega^{\alpha}_{\beta}
&=-\frac{1}{2}K_{\gamma\delta\beta}^{\alpha}e^{\gamma}e^{\delta}-\omega^{\alpha}_{\gamma}\wedge\omega^{\gamma}_{\beta}.\label{structure
equation 2}
\end{align}
We now consider a coordinate system
$\phi=(x^1,\ldots,x^n,y^1,\ldots,y^d)$ around the identity element
$e\in G_0$ such that $dx^{\alpha}_{|e}=e^{\alpha}_{|e}$, and take
$$\begin{array}{rrcl}
f: & \widetilde{\mathcal{U}} & \to & \mathcal{U}\\
   & (a_1,\ldots,a_n) & \mapsto &
   \phi^{-1}(a_1,\ldots,a_n,0,\ldots,0),
\end{array}$$
where $\mathcal{U}$ is the coordinate neighborhood and
$\widetilde{\mathcal{U}}$ is an open subset of $\R^n$ where $f$
can be defined. It is evident that the map $f$ defines an
immersion from an open set $W\subset\R^n$ containing the origin of
$\R^n$ into $G_0$. Let $\widetilde{E}^{\alpha}=f^*(e^{\alpha})$,
since these $1$-forms are linearly independent at the origin of
$\R^n$, there is an open set $M\subset W$ around the origin where
they are linearly independent. Let
$\{\widetilde{E}_1,\ldots,\widetilde{E}_{n}\}$ be the dual frame
field,
 we define on $M$ the pseudo-Riemannian metric
$$g = \sum_{\alpha=1}^n\widetilde{E}^{\alpha}\otimes\widetilde{E}^{\alpha},$$
and the tensor fields
\begin{align*}
\widetilde{T} &=
T_{\alpha\beta}^{\gamma}\widetilde{E}^{\alpha}\otimes
\widetilde{E}^{\beta}\otimes
\widetilde{E}_{\gamma},\\
\widetilde{K} & =
K_{\alpha\beta\gamma}^{\delta}\widetilde{E}^{\alpha}\otimes
\widetilde{E}^{\beta}\otimes \widetilde{E}^{\gamma}\otimes \widetilde{E}_{\delta},\\
\widetilde{P}
&=P_{\alpha_1\ldots\alpha_u}^{\beta_1\ldots\beta_v}\widetilde{E}^{\alpha_1}\otimes\ldots\otimes
\widetilde{E}^{\alpha_u}\otimes
\widetilde{E}_{\beta_1}\otimes\ldots\otimes
\widetilde{E}_{\beta_v}.
\end{align*}
In addition we consider $\widetilde{\omega}=f^*\omega$ which is a
$1$-form on $M$ with values in $\f{h}_0$. Note that
$\{\widetilde{E}_1,\ldots,\widetilde{E}_{n}\}$ is an orthonormal
frame field defined on the whole $M$, so that it is a section of
the bundle of orthonormal frames $\mathcal{O}(M)$ which
trivializes it. Hence, making use of that section,
$\widetilde{\omega}$ is the $1$-form of a metric connection
$\wnabla$ on $\mathcal{O}(M)$. By \eqref{structure equation 1} and
\eqref{structure equation 2}, which are nothing but the structure
equations for the torsion and curvature of $\omega$, we have that
$\widetilde{T}$ and $\widetilde{K}$ are the torsion and curvature
of the connection $\wnabla$ respectively. Since
$\widetilde{\omega}$ takes values in $\f{h}_0$, we have that
$\widetilde{T}$, $\widetilde{K}$ and $\widetilde{P}$ are parallel
with respect to $\wnabla$, that is, $\wnabla$ is an
ASK-connection. Therefore, $(M,g,\widetilde{P})$ is locally
homogeneous with $\widetilde{P}$ invariant. Finally, making use of
Remark \ref{remark recovery of infinitesimal information}, it is
easy to see that the covariant derivatives of $\widetilde{P}$ and
the curvature of $g$ at the origin coincide with
$R^0,\ldots,R^{r+2}$, $P^0,\ldots,P^{s+2}$ under the
identification $T_oM\simeq V$. In addition, by this identification
$M$ is $(r,s)$-strongly reductive.

In both, the first and the second part of the theorem, uniqueness
(up to local isometry) follows from Proposition \ref{proposition
uniqueness pseudo}.
\end{proof}

\bigskip

Note that the strong reductivity condition \eqref{condition
strongly-reductivity} is essential in the proof of Theorem
\ref{reconstruction theorem pseudo-Riemannian}, since otherwise we
are not able to construct the infinitesimal model $(T,K,P)$ from
the infinitesimal data $R^0,\ldots,R^{r+2}$, $P^0,\ldots,P^{s+2}$.
This means that in general a locally homogeneous pseudo-Riemannian
manifold whose metric is not definite might not be recovered from
infinitesimal data. If the manifold admits an ASK-connection
$\wnabla$, this problem can be solved if we add to
$R^0,\ldots,R^{r+2}$, $P^0,\ldots,P^{s+2}$ the knowledge of either
$S_p$, where $S=\wnabla-\nabla$, the torsion of $\wnabla$ at $p$,
or the curvature of $\wnabla$ at $p$ (these three last items
provide equivalent information in view of Remark \ref{remark
recovery of infinitesimal information}). In that case, an
analogous result to Proposition \ref{proposition uniqueness
pseudo} can be proved by a straightforward adaptation.

\section{Examples and the reductivity condition}\label{section reductivity condition}

We begin this section showing a necessary condition for a
reductive locally homogeneous pseudo-Riemannian manifold to be
locally isometric to a globally homogeneous pseudo-Riemannian
manifold. This question has already been solved in the Riemannian
case (see for instance \cite{Ko3} and \cite{Spiro}).

\begin{proposition}\label{proposition regular implica locali isometric to global}
Let $(M,g,\mathcal{G})$ be a reductive locally homogeneous
pseudo-Rie\-mannian manifold endowed with an associated
AS-connection $\wnabla$. If the infinitesimal model associated
with $\wnabla$ is regular, then $(M,g)$ is locally isometric to a
reductive globally homogeneous pseudo-Riemannian manifold. The
same holds if the transvection algebra is regular.
\end{proposition}

\begin{proof}
Let $p\in M$, consider the Nomizu construction $\f{g}_0=T_pM\oplus
\f{h}_0$ associated with the infinitesimal model $(T,K)$. Let
$G_0$ be the simply-connected Lie group with Lie algebra
$\f{g}_0$, and $H_0$ its connected subgroup with Lie algebra
$\f{h}_0$. If $(T,K)$ is regular then $H_0$ is closed in $G_0$, so
that we can consider the homogeneous space $G_0/H_0$. Moreover,
$G_0/H_0$ is reductive as $\f{g}_0=T_pM\oplus \f{h}_0$ is a
reductive decomposition, and the tangent space of $G_0/H_0$ at the
origin $o$ is identified with $T_pM$ through a linear isomorphism
$F:T_pM\to T_0(G_0/H_0)$. This homogeneous space is thus endowed
with a $G_0$-invariant pseudo-Riemannian matric inherited from $g$
at $p$. We consider the canonical connection $\wnabla^{can}$
associated with this reductive decomposition (see \cite[Ch.
X]{KN}). Under the identification $F$, the curvature and torsion
of $\wnabla$ coincides with $K$ and $T$ respectively. This means
that there is a linear isometry $F:T_pM\to T_0(G_0/H_0)$
preserving the curvature and torsion of $\wnabla$ and
$\wnabla^{can}$. Therefore, there are open neighborhoods
$\mathcal{U}$ and $\mathcal{V}$ of $p$ and $o$ respectively, and
an affine map $f:\mathcal{U}\to\mathcal{V}$ with respect to
$\wnabla$ and $\wnabla^{can}$ taking $p$ to $o$ (see \cite[Vol. I,
Ch. VI]{KN}). Since both connections are metric we have that $f$
is an isometry. The same arguments can be applied substituting the
Nomizu construction by the transvection algebra.
\end{proof}

\bigskip

As we know, a globally homogeneous space can be represented as
different coset spaces $G/H$. In the same way, we can consider the
action of different Lie pseudo-groups of isometries on the same
locally homogeneous pseudo-Riemannian manifold $(M,g)$. Since the
notion of reductivity is tied to the action of a Lie pseudo-group
in particular, the following question naturally arises: let
$\mathcal{G}$ and $\mathcal{G}'$ be Lie pseudo-groups of
isometries acting transitively on $(M,g)$, is it possible that
$(M,g,\mathcal{G})$ is reductive but $(M,g,\mathcal{G}')$ is
non-reductive? We now present some examples which give an
affirmative answer to this question, and explores the possible
scenerarios when $\mathcal{G}$ is a subgroup of $\mathcal{G}'$ and
viceversa. We will also show that the reductivity condition does
not imply the strong reductivity condition. It is worth pointing
out that this situation is not a consequence of the freedom
obtained by enlarging the (rather rigid) family of globally
homogeneous spaces to the family of locally homogeneous spaces,
and we can find illustrative examples restricting ourselves to
globally homogeneous pseudo-Riemannian manifolds. We will finally
give an example of an stabilizing pair distinct of $(k(p),l(p))$
(see Remark \ref{remark stabilizing pair}).


\begin{example}
Consider $\R^5$ endowed with the standard metric $\eta$ of
signature $(2,3)$. We take the $4$-dimensional submanifold
$$\mathbb{H}^4_1=\{x\in\R^5/\, \eta(x,x)=-1\},$$
endowed with the pseudo-Riemannian metric $g$ inherited from
$\eta$. $(\mathbb{H}^4_1,g)$ is a Lorentz space of constant
sectional curvature, and it is well known that it is the
(globally) symmetric space
$$\mathbb{H}^4_1\simeq \frac{\Lie{SO}_0(2,3)}{\Lie{SO}_0(1,3)}.$$
Let $\{e_1,\ldots,e_5\}$ be the standard basis of $\R^5$, and let
$e^j_i$ denote the endomorphism $e^j\otimes e_i$ of $\R^5$. The
isotropy algebra at the point $p=(0,1,0,0,0)\in\mathbb{H}^4_1$ is
$$\f{so}(1,3)=\mathrm{Span}\{e_1^3+e_3^1,e_1^4+e_4^1,e_1^5+e_5^1,e_3^4-e_4^3,e_3^5-e_5^3,e_4^5-e_5^4\}.$$
An $\Lie{SO}_0(1,3)$-invariant complement is
$$\f{m}=\mathrm{Span}\{e_1^2-e_2^1,e_2^3+e_3^2,e_2^4+e_4^2,e_2^5+e_5^2\},$$
hence $(\mathbb{H}^4_1,g,\Lie{SO}_0(2,3))$ is reductive. Consider
now the Lie subalgebra $\f{g}$ spanned by the elements
$$
e_1^4+e_4^1-e_2^3-e_3^2,\,\dfrac{1}{2}(e_1^2-e_2^1+e_1^3+e_3^1+e_2^4+e_4^2+e_4^3-e_3^4),$$
$$
\dfrac{1}{2}(e_1^3+e_3^1+e_2^1-e_1^2+e_2^4+e_4^2+e_3^4-e_4^3),\,\dfrac{1}{2}(e_2^1-e_1^2+e_2^4+e_4^2+e_4^3-e_3^4-e_1^3-e_3^1),$$
$$
\dfrac{1}{\sqrt{2}}(e_1^5+e_5^1+e_5^4-e_4^5),\,\dfrac{1}{\sqrt{2}}(e_3^5-e_5^3-e_2^5-e_5^2),\,e_1^4+e_4^1+e_2^3+e_3^2.$$
The isotropy algebra $\f{k}$ at $p$ is spanned by the elements
$$2(e_1^4+e_4^1+e_1^4),\,e_1^3+e_3^1+e_3^4-e_4^3,\,\dfrac{1}{\sqrt{2}}(e_1^5+e_5^1+e_5^4-e_4^5).$$
Let $G$ be the connected Lie subgroup of $\Lie{SO}_0(2,3)$ with
Lie algebra $\f{g}$, then $G$ acts transitively on
$\mathbb{H}^4_1$, but there is no $\mathrm{ad}(\f{k})$-invariant
complement of $\f{k}$, so that $(\mathbb{H}^4_1,g,G)$ is
non-reductive (see Lie algebra $A5^*$ in \cite{Fels-Renner}).
\end{example}

\begin{example}\label{example B3}
We consider $\R^4$ endowed with the pseudo-Riemannian metric

\bigskip

\noindent
\resizebox{\linewidth}{!}{$g=2e^{y_1}\cos{y_2}(dy_1dy_4-dy_2dy_3)-2e^{y_1}\sin{y_2}(dy^1dy^3+dy^2dy^4)+Le^{4y_1}dy_2dy_2,$}

\bigskip

\noindent with $L\in\R-\{0\}$. Let $\widetilde{\Lie{SL}(2,\R)}$ be
the universal cover of $\Lie{SL}(2,\R)$, the group
$G'=\widetilde{\Lie{SL}(2,\R)}\ltimes\R^2\times\R$ acts
transitively by isometries on $(\R^4,g)$ (see \S 5 of
\cite{Fels-Renner}). The Lie algebra of $G'$ can be written as
$$[e_1,e_2]=2e_2,\qquad [e_1,e_3]=-2e_3,\qquad [e_2,e_3]=e_1,\qquad [e_1,e_4]=e_4,$$
$$[e_1,e_5]=-e_5,\qquad [e_2,e_5]=e_4, \qquad [e_3,e_4]=e_5,$$
with respect to some basis $\{e_1,\ldots,e_6\}$. It can be found
as the Lie algebra $B3$ in \cite{Fels-Renner}, and moreover it is
the full isometry algebra of $(\R^4,g)$ and can be realized by the
complete Killing vector fields
\begin{align*}
Y_1 & =
\cos(2y_2)\partial_{y_1}-\sin(2y_2)\partial_{y_2}+y_3\partial_{y_3}-y_4\partial_{y_4},\\
Y_2 &= \frac{1}{2}\sin(2y_2)\partial_{y_1}+\cos^2(y_2)\partial_{y_2}+y_3\partial_{y_4},\\
Y_3 &=\frac{1}{2}\sin(2y_2)\partial_{y_1}-\sin^2(y_2)\partial_{y_2}+y_4\partial_{y_3},\\
Y_4 &=\partial_{y_4},\\
Y_5 &=-\partial_{y_3},\\
Y_6
&=e^{y_1}\cos(y_2)\partial_{y_3}+e^{y_1}\sin(y_2)\partial_{y_4}.
\end{align*}
The isotropy algebra at $(0,0,0,0)\in\R^4$ is
$\mathrm{Span}\{e_3,e_5+e_6\}$. As stated in \cite{Fels-Renner},
$(\R^4,g,G')$ is non-reductive. Let
$\f{g}=\mathrm{Span}\{e_1,e_2,e_4,e_5,e_6\}$. Making use of the
distribution generated by the corresponding Killing vector fields
we see that the action of the connected Lie subgroup $G\subset G'$
with Lie algebra $\f{g}$ is still transitive. The isotropy algebra
at $(0,0,0,0)$ is $\f{k}=\mathrm{Span}\{e_5+e_6\}$, and
$\f{m}=\mathrm{Span}\{e_1,e_2,e_4,e_5\}$ is an
$\mathrm{Ad}(K)$-invariant complement, where $K\subset G$ is the
isotropy group with respect to the action of $G$ at $(0,0,0,0)$.
Therefore $(\R^4,g,G)$ is reductive. On the other hand we can
check that $(\R^4,g)$ is not strongly reductive. In this case,
since there is no extra geometric structure, the complex of
filtrations reduces to
$$\f{so}(T_pM)\supset \f{g}(p,0)\supset \f{g}(p,1)\supset\ldots$$
A simple computation shows that the only non-zero component of the
curvature is
$R_{\partial_{y_1}\partial_{y_2}\partial_{y_1}\partial_{y_2}}=-3Le^{4y_1}$,
and $\nabla R=0$. We take $p=(0,0,0,0)$ and $L=1$ for the sake of
simplicity, so that the filtration actually is
$$\f{so}(T_pM)\supset \f{g}(p,0)=\f{g}(p,1),$$
where \begin{align*}\f{so}(T_pM)&=\left\{\begin{pmatrix}-e & 2(b-c) & b & 0\\
f & 2a & a & c \\ 2(d-f) & 0 & -2a & 2(b-c)\\ 0 & 2(d-f) & d & e
\end{pmatrix}/\,a,b,c,d,e,f\in\R \right\},\\
\f{g}(p,0) &=\left\{A\in\f{so}(T_pM)/\, e=2a,\,
f=d\right\}.\end{align*} It is easy to check that $\f{g}(p,0)$
does not admit any complement $\f{n}$ invariant by the adjoint
action of $\f{g}(p,0)$, hence $(\R^4,g)$ cannot be strongly
reductive.
\end{example}

We now exhibit an example of a locally homogeneous pseudo-K\"ahler
manifold with an stabilizing pair distinct form $(k,l)$, where as
usual $(k,l)$ are the first integers such that
$\f{g}(p,k)=\f{g}(p,k+1)$ and $\f{p}(p,l)=\f{p}(p,l+1)$.

\begin{example}
Consider the space $\C^{2}$ with complex coordinates $(w,z)$. We
take $M=\C^{2}-\{||w||=0\}$ with the standard complex structure
$J$ and the pseudo-Riemannian metric
$$g=dw^1dz^1+dw^2dz^2+b(dw^1dw^1+dw^2dw^2),$$
where $w=w^1+iw^2$, $z=z^1+iz^2$, and $b$ is a function depending
on $w^1$ and $w^2$ and satisfying $\Delta b=\frac{b_0}{||w||^4}$.
This manifold is locally homogeneous since it admits an
ASK-connection \cite{Cas-Lu}. Let
$\theta=-\frac{1}{||w||^2}(w^1dw^1+w^2dw^2)$, the curvature tensor
and its first covariant derivative are
$$R=\frac{1}{2}\frac{b_0}{||w||^4}(dw^1\wedge dw^2\otimes dw^1\wedge dw^2), \qquad \nabla R=4\theta\otimes R.$$
We set $b_0=2$ and take the point $p=(-1,0,0,0)$, so that
$$R_p=dw^1\wedge dw^2\otimes dw^1\wedge dw^2,\qquad \nabla R_p=4dw^1\otimes R_p,$$
$$\nabla^2 R_p=(20dw^1\otimes dw^1-4dw^2\otimes dw^2)\otimes R_p.$$
On the other hand, $J_p$ is the standard complex structure of
$\C^{2}$ and $\nabla J_p=0$ since the manifold is pseudo-K\"ahler.
A straightforward computations thus shows that the complex of
filtrations is
$$\begin{matrix}
\f{so}(\R^4)^6 & \supset & \f{g}(p,0)^2 & \supset & \f{g}(p,1)^1 &
= & \f{g}(p,2)^1\\
\cup & & || & & || & & || \\
\f{p}(p,0)^4 & \supset & \f{h}(p,0,0)^2 & \supset & \f{h}(p,1,0)^1
&
= & \f{h}(p,2,0)^1\\
|| & & || & & || & & || \\
\f{p}(p,1)^4 & \supset & \f{h}(p,0,1)^2 & \supset & \f{h}(p,1,1)^1
& = & \f{h}(p,2,1)^1,
\end{matrix}$$
where superindexes indicate dimension. We have that $(k,l)=(1,0)$,
but $(r,s)=(1,-1)$ is a stabilizing pair.
\end{example}

\section*{Acknowledgements}

The author is deeply indebted to Prof. M. Castrill\'on L\'opez and
Prof. P.M. Gadea for useful conversations about the topics of this
paper. This work has been partially funded by MINECO (Spain) under
project MTM2011-22528.

\end{document}